\documentclass[twoside, 11pt]{article}
\usepackage{amsmath,amssymb,amsthm,mathrsfs}
\usepackage[margin=2.5cm]{geometry} 
\usepackage{lipsum}
\usepackage{titlesec,hyperref}
\usepackage{fancyhdr}
\usepackage{graphicx}
\usepackage[numbers,sort&compress]{natbib}
\usepackage{color}

\fancypagestyle{plain}
{
\fancyhf{}

}

\titleformat{\subsection}{\it}{\thesubsection.\enspace}{1.5pt}{}
\titleformat{\subsubsection}{\it}{\thesubsubsection.\enspace}{1.5pt}{}

\newtheorem{theo}{Theorem}[section]
\newtheorem{lemm}[theo]{Lemma}
\newtheorem{defi}[theo]{Definition}

\newtheorem{rema}{Remark}[section]
\numberwithin{equation}{section}

\allowdisplaybreaks

\def\th2{\frac{\theta}{2}}

\begin{document}
\title{ Effect of weak elasticity on the  Kelvin-Helmholtz instability \hspace{-4mm}}

	\author{ Binqiang Xie $^{\dag}$, Boling Guo$^{\ddag}$ , Bin Zhao $^{\ddag,*}$ \\[10pt]
	\small {$^\dag$ School of Mathematics and Statistics,}\\
	\small { Guangdong University of Technology, Guangzhou,   510006, China}
	\\
	\small {$^\ddag$  Institute of Applied Physics and Computational Mathematics, Beijing, 100088,  China}
}

\footnotetext{*Corresponding author.\\ E-mail addresses:  \it zhaobin2017math@163.com(B. Zhao), \it xbqmath@gdut.edu.cn(B.Q. Xie), \it  gbl@iapcm.ac.cn(B.L. Guo).}

\date{}

\maketitle

\begin{abstract}
In this paper, we present an analysis of the Kelvin-Helmholtz instability in two-dimensional ideal 
 compressible elastic flows, providing a rigorous confirmation that weak elasticity has a destabilizing effect on the Kelvin-Helmholtz instability. There are two critical velocities, $U_{\text{low}}$ and $U_{\text{upp}}$, where $U_{\text{low}}$ and $U_{\text{upp}}$ represent the lower and upper critical velocities, respectively. We demonstrate that when the magnitude of the rectilinear solutions  satisfies $U_{\text{low}}+c\epsilon_{0}\le  |\dot{v}^{+}_{1}| \le  U_{\text{upp}}-c\epsilon_{0}$, the linear and nonlinear ill-posedness of the piecewise smooth solutions of the Kelvin-Helmholtz problem for two-dimensional ideal compressible elastic fluids is established uniformly, where $c$ is the sound speed and $\epsilon_{0}$ is some small enough positive constant.

\vspace*{5pt}

\vspace*{5pt}
\end{abstract}


\section{Introduction}
\quad
In this paper, we will prove that weak elasticity has a destabilizing effect on the Kelvin-Helmholtz instability for the ideal compressible elastic fluids.  The two-dimensional compressible inviscid elastodynamics in the domain $\mathbb{R}^{2}$ for time $t\geq0$  is considered as the following form:
\begin{equation}\label{1.1}
	\left\{
	\begin{aligned}
		&\partial_{t} \rho + {\rm div} (\rho u)=0,\\
		&\partial_{t}(\rho v) + {\rm div}(\rho u \otimes u) + \nabla p={\rm div} (\rho F F^{T}),\\
		&\partial_{t}(\rho F_{j})+ {\rm div}( u\otimes \rho F_{j}-\rho F_{j}\otimes u )=0,
	\end{aligned}
	\right.
\end{equation}
where the functions $\rho$, $u$, $p$ and $F_{j}$
represent the fluid density, the velocity, the pressure and the $j$th column of the deformation gradient $F_{ij}, i,j=1,2$ respectively, $p$ is a $C^{\infty}$ function of $\rho$, defined on $(0,+ \infty)$, and such that $p^{\prime}(\rho)>0$ for all $\rho$. The speed of sound $c(\rho)$  in the fluid is defined by the relation:
\begin{equation}\label{1.2}
	\forall \rho >0, ~c(\rho)= \sqrt{p^{\prime}(\rho)}.
\end{equation}
The system \eqref{1.1} is supplemented by the divergence constraint
\begin{equation}\label{1.3}
	{\rm div}(\rho F_{j})=0,~ j=1,2
\end{equation}
and this property holds at any time if it is satisfied initially.

Assume $U(t,x_{1},x_{2})= (\rho,u, F)(t,x_{1},x_{2})$ to be the solution of the system \eqref{1.1} which is smooth on each side of  the surface  $\Gamma(t):= \{x_{2}=f(t,x_{1}) \}$, here function $f$ describing the discontinuity front is part of the unknown of the problem, i.e. this is a free boundary problem and $x_{1}$ is tangential coordinate. The whole space $\mathbb{R}^{2}$ can be seperated by $\Gamma(t)$ into the upper domain $\Omega^+(t)$ and the lower domain $\Omega^-(t)$, which are defined by $$\Omega^+(t):=\{x_2>f(t,x_1)\},$$
and  $$\Omega^-(t):=\{x_2<f(t,x_1)\}.$$\\
We denote the solutions of the upper and lower fluids by
\begin{equation} \label{1.4}
	U=\left\{
	\begin{aligned}
		& U^{+}(t,x_{1},x_{2}),  &\text{in}~ \Omega^+(t) ,\\
		&U^{-}(t,x_{1},x_{2}),&\text{in}~ \Omega^-(t),
	\end{aligned}
	\right.
\end{equation}
where $U^{\pm}= (\rho^{\pm},u^{\pm}, F^{\pm})$. As we are interested in the smooth solutions of \eqref{1.1} on either side of $\Gamma(t)$, such piecewise smooth solutions $U$ to be weak solutions of \eqref{1.1} should satisfy the  Rankine-Hugoniot conditions:
\begin{equation}\label{1.5}
\left\{
\begin{aligned}
&\partial_{t} f [\rho] - [\rho u \cdot n]=0,\\
&\partial_{t}  f  [\rho u]- [(\rho u \cdot n) u]- [p] n + [\rho F F^{T} n]=0,\\
&\partial_{t} f [\rho F_{j}] - [(\rho u \cdot n)\rho F_{j}]+ [(\rho F_{j}\cdot n)u]=0,\\
&[\rho F_{j}\cdot n]=0,
\end{aligned}
\right.
\end{equation}
where the notation $[\phi]= \phi^{+}|_{\Gamma(t)}-\phi^{-}|_{\Gamma(t)}$ denotes the jump of a quantity $\phi$ across $\Gamma(t)$, $n=(-\partial_{1} f, 1)$  is a normal vector to $\Gamma(t)$ with $\partial_{1} = \partial/\partial x_{1}$.  If we denote mass transfer flux by $S^{\pm}=\rho^{\pm} (\partial_{t} f-  u^{\pm} \cdot n)$, in according with the first condition in \eqref{1.5}, we get $[S]=0$. In order to induce  Kelvin-Helmholtz instability, we assume that  $S^{\pm}=0$ and $F^{\pm}_{j}\cdot n=0$ on $\Gamma(t)$.
Therefore, for   Kelvin-Helmholtz instability,  the Rankine-Hugoniot conditions \eqref{1.5} give the boundary conditions
\begin{equation}\label{1.6}
	p^{+}=p^{-}, ~~\partial_{t} f=u^{+}\cdot n=u^{+}\cdot n,~~F^{+}_{j} \cdot n=F^{-}_{j} \cdot n=0~~on~~\Gamma(t).
\end{equation}

To complete the statement of the problem, we specify initial conditions. For the initial front $f(0)$,  we give the initial interface $\Gamma(0)=\{x_{2}=f(0,x_{1})=f_{0}(x_{1})\}$, which yields the open sets $\Omega^{\pm}(0)$ on which we specify the initial density, velocity and  deformation gradient, $\rho^{\pm}(0)=\rho^{\pm}_{0}: \Omega^{\pm}(0) \rightarrow \mathbb{R}^{+}$, $u^{\pm}(0)=u^{\pm}_{0}: \Omega^{\pm}(0) \rightarrow \mathbb{R}^{2}$ and  $F^{\pm}(0)=F^{\pm}_{0}: \Omega^{\pm}(0) \rightarrow \mathbb{R}^{2\times 2}$, respectively.

Because $p^{\prime}(\rho)>0$, the function $p=p(\rho)$ can be inverted, allowing us to write $\rho=\rho(p)$. Given a positive constant $\dot{\rho}>0$, we introduce the quantity $\sigma(p)=\log(\rho(p)/\dot{\rho})$ and consider $\sigma$ as a new unknown quantity.  In terms of $(\sigma,u,F)$, the equations \eqref{1.1} become
\begin{equation}\label{1.8}
	\begin{cases}
		\partial_{t} \sigma+\left(u \cdot \nabla\right) \sigma+  \nabla\cdot u=0, \\ 
		\partial_{t} u+\left(u \cdot \nabla\right) u+ c^{2} \nabla \sigma=\sum_{j=1}^{2} (F_{j}\cdot \nabla ) F_{j} , \\
		\partial_{t} F_{j}+\left(u \cdot \nabla\right) F_{j}=  \left(F_{j} \cdot \nabla \right) u,
	\end{cases}
\end{equation}
where the speed of sound is considered as a function of $\sigma$, i.e., $c=c(\sigma)$.

The jump conditions \eqref{1.6} may be rewritten as 
\begin{equation} \label{1.9}
	u^{+}\cdot n= u^{-}\cdot n, ~\sigma^{+}= \sigma^{-},~F^{+}_{j} \cdot n=F^{-}_{j} \cdot n=0~\mathrm{on}~~\Gamma(t).
\end{equation}

\subsection{Rectilinear solution}

\quad  It is easy to see that the system \eqref{1.1}-\eqref{1.6} admits contant solutions $\dot{U}=(\dot{f}, \dot{\rho},\dot{u},\dot{F})$ with the corresponding interface satisfying $\Gamma:=\{x_{2}=0\}$ for all $t\geq0$. Then the corresponding domain satisfies $\Omega^{+}=\Omega^{+}(t)=\mathbb{R}\times (0,\infty)$ and  $\Omega^{-}=\Omega^{-}(t)=\mathbb{R}\times (-\infty,0)$ for all $t\geq0$.  More precisely, the front is flat, i.e., $\dot{f}=0$. To make sure the constant   density $\dot{\rho}^{\pm}$ satisfy the jump condition \eqref{1.5}, we must impose that 
\begin{equation} \label{1.10}
	\dot{\rho}^{+}=	\dot{\rho}^{-}=:\dot{\rho},
\end{equation}
where $\dot{\rho}$ is a positive constant. Under a Galiean transformation, we may assume that the upper fluid moves in the horizontal direction with some constant velocity and the lower fluid moves by the same constant velocity in the opposite direction, i.e, the  constant velocity field $\dot{u}$  is the following form:
\begin{equation}\label{1.11}
	\dot{u}=\left\{
	\begin{aligned}
		&(\dot{u}^{+}_{1},0)&x_2\geq0,\\
		&(\dot{u}^{-}_{1},0)&x_2<0,
	\end{aligned}
	\right.
\end{equation}\label{1.12}
where the constants $\dot{u}^{+}_{1},\dot{u}^{-}_{1}$ satisfy
\begin{equation}
	\dot{u}^{+}_{1}=-\dot{u}^{-}_{1}.
\end{equation}

Under the change of the scale of measurement, the constant deformation gradient $\dot{F}$ is the following form:
\begin{equation} \label{1.13}
	\dot{F}=\left\{
	\begin{aligned}
		&\left[ {\begin{array}{cc}\dot{F}^{+}_{11} & \dot{F}^{+}_{12} \\0 & 0 \\\end{array} } \right]&x_2\geq0,\\
		&\left[ {\begin{array}{cc}\dot{F}^{-}_{11} & \dot{F}^{-}_{12} \\0 & 0 \\\end{array} } \right]&x_2<0,
	\end{aligned}
	\right.
\end{equation}
where the constants $\dot{F}^{+}_{11},\dot{F}^{+}_{12},\dot{F}^{-}_{11},\dot{F}^{-}_{12}$  satisfy
\begin{equation}\label{1.14}
	\dot{F}^{+}_{11}=-\dot{F}^{-}_{11},~	\dot{F}^{+}_{12}=-\dot{F}^{-}_{12}.
\end{equation}

\subsection{History result}
\quad In Chandrasekhar's book \cite{Chandrasekhar}, the stability problem of superposed fluids can be divided into two kinds, the first kind of instability is called Rayleigh-Taylor instability. There are lot of works about mathematical analysis of the Rayleigh-Taylor instability problem (\cite{Castro}, \cite{Guo11}, \cite{Guo1},\cite{Guo2},\cite{Guo3}, \cite{Jiang}). Ebin in \cite{E2} proved the instability for the Rayleigh-Taylor problem of  the incompressible Euler equation, while Guo and Tice in \cite{Guo1} showed the instability of this problem for the compressible inviscid case. Moreover, the Rayleigh-Taylor instability for the viscous compressible fluids was proved in \cite{Guo2} and for the inhomogeneous Euler equation in \cite{Guo3}. The second type of instability arises when the different layer of stratified heterogeneous fluids are in relative horizontal motion. In this paper, we study the second kind.

\par The stability problem of two fluids in a relative motion has attracted a wide interest of researchers of various fields. This type of instability is well known as the Kelvin-Helmholtz instability which was first studied
by Hermann von Helmholtz in  \cite{Helmholtz} and by William Thomson (Lord Kelvin) in \cite{Kelvin}. The Kelvin-Helmholtz instability is important in understanding a variety of space and astrophysical phenomena involving sheared plasma flow such as the stability of the
interface between the solar wind and the magnetosphere (\cite{Dungey},\cite{Ha}, \cite{Parker}), interaction between adjacent streams of different velocities in the solar wind \cite{Sturrock}  and the dynamic structure of cometary tails \cite{Ershkovich}.

There are a lot of  progress on the well-posedness of solutions for the Kelvin-Helmholtz problem of the ideal fluids. The Kelvin-Helmholtz instability configuration is also  known in literature as the ‘vortex sheets’, as their vorticity distribution is described by a $\delta$-function supported by a discontinuity in the velocity field at the sheet location. In the pioneer works \cite{Coulombel}, \cite{Secchi}, Coulombel and Secchi proved the  nonlinear stability of vortex sheets  for the ideal compressible flows by using a micro-local analysis and Nash-Moser method. Later on, Morando, Trebeschi and Wang \cite{Trebeschi1}, \cite{Trebeschi2} generalized this result to the two-dimensional ideal nonisentropic compressible flows. Their method in \cite{Coulombel} also has been used to deal with
the two-dimensional magnetohydrodynamics (MHD) flows, a necessary and sufficient condition is obtained for the linear stability of the rectilinear vortex sheets in Wang and Yu \cite{Yu1}.  Moreover, for the three-dimensional compressible magnetohydrodynamics flows, Trakhinin \cite{Trakhinin2}, \cite{Trakhinin} and Chen-Wang \cite{Wang} adopted a different symmetrization approach to prove the  linear and nonlinear stability of compressible vortex sheets. These results  indicate the stabilization effects of the magnetic fields on the current vortex sheets. On the other hand, these  stabilization effects of the magnetic fields on the current vortex sheets are also shown in the incompressible case.  Under the Syrovatskij stability condition \cite{Syrovatskij}, Morando, Trakhinin, and Trebeschi \cite{Morando2} proved an a priori estimate with a loss of three-order derivatives for the linearized system. Trakhinin \cite{Trakhinin} proved an a priori estimate without loss of derivative from data for the linearized system with variable coefficients under a strong stability condition. In a recent work \cite{Coulombel2}, Coulombel, Morando, Secchi, and Trebeschi proved an
a priori estimate without loss of derivatives for the nonlinear current-vortex sheet problem under the strong stability condition. Nonlinear stability of the incompressible current-vortex sheet problem was proved by Sun-Wang-Zhang under Syrovatskij stability condition in \cite{Sun}.  For vortex sheets  in elastic flows,  Chen, Hu and Wang (\cite{Hu}, \cite{Chen})  proved strong elasticity can inhibit  Kelvin-Helmholtz  vortices. Further, Chen, Huang, Wang and Yuan proved the stabilization effect of elasticity on three-dimensional compressible vortex sheets in \cite{chen1}.

 Meanwhile, there are some progress on the ill-posedness of solutions for the Kelvin-Helmholtz problem of the ideal fluids.  For Kelvin-Helmholtz instability in the incompressible Euler flows,  Ebin in \cite{E2} proved linear and nonlinear ill-posedness of the  well-known Kelvin-Helmholtz problem. Recently we prove  linear and nonlinear ill-posedness of the  Kelvin-Helmholtz problem for incompressible MHD fluids \cite{Xie} under the condition violating the Syrovatskij stability condition. On the other hand, for Kelvin-Helmholtz instability in the compressible Euler flows.  It has been well-known by Landau \cite{L} that the Kelvin-Helmholtz instability is suppressed in compressible supersonic flows. By the normal  mode analysis, it is also shown in \cite{Fejer}, \cite{Miles} that the linear  Kelvin-Helmholtz instability can be inhibitied  when the Mach number  $M:=\frac{|\dot{v}^{+}_{1}|}{c}>\sqrt{2}$ and the solutions of the linear equation are violently unstable when $M< \sqrt{2}$.  Our work \cite{Xie1} proved ill-posedness of Kelvin-Helmholtz problem  for  the nonlinear  Euler fluids
exhibit the same ill-posedness as their linearized counterparts in \cite{Fejer}, \cite{Miles} under the condition $\epsilon_{0}\leq M\le \sqrt{2}-\epsilon_{0}$, where $\epsilon_{0}$ is a small but fixed number.
 
Our aim of this paper is  to prove weak elasticity has a destabilizing effect on the Kelvin-Helmholtz instability of the two-dimensional ideal compressible elastic fluids. By the eigenvalue analysis for the linearized system to the Kelvin-Helmholtz problem of the two-dimensional ideal compressible elastic fluids, we show that  the front will grow instantaneously with time envloves for the high frequency case when $U_{low}<|\dot{v}^{+}_{1}|<U_{upp}$.  Inspired by the works \cite{Guo1}, \cite{Morando} and \cite{Xie1}, we prove the
 ill-posedness of the nonlinear system \eqref{1.1}-\eqref{1.6} uniformly when the rectilinear solutions  satisfy $U_{low}+c\epsilon_{0}\le |\dot{v}^{+}_{1}|\le U_{upp}-c\epsilon_{0}$ with $U_{low}:=\sqrt{((\dot{G}^{+}_{11})^{2}+(\dot{G}^{+}_{12})^{2})}$ and $U_{upp}:=\sqrt{2c^{2}+ ((\dot{G}^{+}_{11})^{2}+(\dot{G}^{+}_{12})^{2})}$.

The rest of the paper is organized as follows. In Section 2, we reformulate the system in a new coordinate and derive the wave equation of the pressure. Besides, we state the main result in the flatting coordinate. In Section 3, we give an analysis of the root for the symbol of the linearized system to the Kelvin-Helmholtz problem of the two-dimensional ideal compressible elastic fluids, which help us identify the instability condition that the rectilinear solutions need satisfy. In Section 4 and Section 5,  we prove the linear  and nonlinear ill-posedness of Kelvin-Helmholtz problem for the ideal compressible elastic fluids.

\section{The new formulation and main result}
\subsection{The new formulation}
\quad  Our analysis in this paper relies on the reformulation of the problem under consideration in new coordinates.  To begin with, we define the fixed domains $\Omega^{\pm}$ as
\begin{equation}\label{2.1}
	\begin{aligned}
		&\Omega^{+}:=\left\{x\in \mathbb{R}^{2}: x_{2}> 0\right\}, \\
		&\Omega^{-}:=\left\{x\in \mathbb{R}^{2}: x_{2}< 0\right\}.
	\end{aligned}
\end{equation}
Define the fixed  boundary $\Gamma$ as
$$
\Gamma:=\left\{x\in \mathbb{R}^{2}: x_{2}= 0\right\}.
$$To reduce our free boundary problem to the fixed domain $\Omega^{\pm}$, we consider a change of variables by $(t,x)\mapsto (t,x_{1}, x_{2}+ \psi(t,x))$ with $x=(x_1,x_2)$ which maps the fixed domains into the free boundary domains.   We   construct such  $\psi$ via multiplying  the front $f$ by a smooth cut-off function depending on $x_{2}$:
\begin{equation}\label{2.2}
	\psi(t,x_{1}, x_{2})= \theta(\frac{x_{2}}{3(1+a)})f(t,x_{1}),~a=\|f_{0}\|_{L^{\infty}(\mathbb{R})},
\end{equation}
where $\theta\in C^{\infty}_{c}(\mathbb{R})$ is a smooth cut-off function with $0\leq \theta\leq 1$, $\theta(x_{2})=1$, for $|x_{2}|\leq 1$, $\theta(x_{2})=0$ for $|x_{2}|\geq 3$, and $ |\partial_{2} \theta(x_{2})|\leq 1$ for all $x_{2}\in \mathbb{R}$,  writing $\partial_{j} = \partial/\partial x_{j}$.   We also assume 
\begin{equation}\label{2.3}
	\|f_{0}\|_{L^{\infty}(\mathbb{R})}\leq 1.
\end{equation}
Moreover, we have
\begin{equation}\label{2.4}
	\begin{aligned}
		&\psi(t,x_{1}, 0)=f(t,x_{1}), \\
		&\partial_{2} \psi(t,x_{1}, 0)=0,\\
		&|\partial_{2} \psi | \leq \frac{1}{3(1+a)}|f|.
	\end{aligned}
\end{equation}

The change of variables that reduces the free boundary problem \eqref{1.1} to the fixed domain $\Omega^{\pm}$ is given in the following lemma.
\begin{lemm}
	Define the function $\Psi$ by
	\begin{equation}\label{2.5}
		\Psi(x_1, x_2):=\left(x_{1}, x_{2}+\psi(t, x)\right), \quad(t, x) \in[0, T] \times \Omega.
	\end{equation}
	Then  $\Psi: (x_1,x_2)\mapsto (x_{1}, x_{2}+ \psi(t,x)) $ is a diffeomorphism transform from  $\Omega^{\pm}$ to $\Omega^{\pm}(t)$ for all $t \in[0, T]$. In particular, we emphasis that $\Psi|_{t=0}: \bar{\Omega}^{\pm }\mapsto \bar{\Omega}^{\pm }(0) $ is a diffeomorphism transform. 
\end{lemm}
\begin{proof}
	Since $\|f_{0}\|_{L^{\infty}(\mathbb{R})}\leq 1$, one can prove that there exists some $T>0$ such that $ \sup_{[0,T]} \|f\|_{L^{\infty}}<2$, the free interface is still a graph within the time interval $[0,T]$	and
	$$
	\begin{aligned}
		\partial_{2} \Psi_{2}(t, x) & =1+\partial_{2} \psi(t, x) \geq 1-\frac{1}{3}\times  2 = \frac{1}{3},
	\end{aligned}
	$$
	which ensure that $\Psi: (x_{1}, x_{2})\mapsto (x_{1}, x_{2}+ \psi(t,x)) $ is a diffeomorphism transform from  $\Omega^{\pm}$ to $\Omega^{\pm}(t)$ for all $t \in[0, T]$.
\end{proof}

We introduce the following operator notation
$$
\begin{array}{ll}
	A=[D \Psi]^{-1}=\left(\begin{array}{ccc}
		1 & 0  \\
		-\partial_{1} \psi / J  & 1 / J
	\end{array}\right), 
\end{array}
$$
$$
a=J A =\left(\begin{array}{ccc}
	J & 0  \\
	-\partial_{1} \psi &  1
\end{array}\right), 
$$ and $ J=\operatorname{det}[D \Psi]=1+\partial_{2} \psi$.
Now we may reduce the free boundary problem \eqref{1.8}-\eqref{1.9} to a problem in the fixed domain $\Omega^{\pm}$ by the change of variables in Lemma 1.1. Let us set 
\begin{equation}\label{2.6}
	\begin{aligned}
		&v^{ \pm}(t, x):=u^{ \pm}(t, \Psi(t, x)), G^{ \pm}(t, x):=F^{ \pm}(t, \Psi(t, x)), \\
		&q^{ \pm}(t, x):=p^{ \pm}(t, \Psi(t, x)),\quad \varrho^{ \pm}(t, x):=\rho^{\pm}(t, \Psi(t, x)),\\
		&\quad h^{ \pm}(t, x):=\sigma^{\pm}(t, \Psi(t, x)).
	\end{aligned}
\end{equation}
Throughout the rest paper, an equation on $\Omega$ means that the equation holds in both $\Omega^{+}$ and $\Omega^{-}$.  For convenience, we consolidate notation by writing $v$, $G_{j}$, $q$, $\varrho$, $h$ to refer to $v^{\pm}$, $G_{j}^{\pm}$, $q^{\pm}$,$\varrho^{\pm}$, $h^{\pm}$ except when necessary to distinguish the two.

We also introduce the notation:
\begin{equation}\label{2.7}
	\nabla^{\psi}= A^{T} \nabla, \Delta^{\psi}= (A^{T} \nabla )\cdot  (A^{T} \nabla).
\end{equation}
Then equations \eqref{1.8}  and boundary condition \eqref{1.9}  can be reformulated as: 
\begin{equation}\label{2.8}
	\begin{cases}
		\partial_{t} h+(\breve{v}\cdot \nabla) h+  \nabla^{\psi} \cdot v=0 & \text { on } \Omega,\\ 
		\partial_{t} v+(\breve{v} \cdot \nabla) v+ c^{2} \nabla^{\psi} h=\sum_{j=1}^{2} (\breve{G}_{j}\cdot \nabla ) G_{j}  & \text { on } \Omega,\\
		\partial_{t} G_{j}+\left(\breve{v} \cdot \nabla\right) G_{j} =(\breve{G}_{j}\cdot \nabla ) v & \text { on } \Omega, \\ 
		\nabla^{\psi} \cdot (\varrho G_{j})=0 & \text { on } \Omega, \\ 
		\partial_{t} f=v\cdot n & \text { on } \Gamma,  \\
		[v\cdot n]=0,~\quad[h]=0,~\quad G_{j}\cdot n=0 & \text { on } \Gamma.
	\end{cases}
\end{equation}
with the initial data
\begin{equation*}
	\begin{cases}
	h|_{ t=0}=h_{0},\quad  v|_{ t=0}=v_{0}, \quad G_{j}|_{ t=0}=G_{0,j} & \text {on } \Omega, \\ 
		f|_{ t=0}=f_{0} & \text { on } \Gamma .
	\end{cases}
\end{equation*}
where  we have define
\begin{equation}\label{2.9}
	\begin{aligned}
		&	\breve{v} :=A v-\left(0, \partial_{t} \psi / J\right)=\left(v_{1}, \left(v \cdot n-\partial_{t} \psi\right) / J\right),\\
		& \breve{G}_{j} :=A G_{j}=\left(G_{1j}, G_{j} \cdot n / J\right).
	\end{aligned}
\end{equation}
Notice that
\begin{equation}\label{2.10}
	J=1, \quad	\breve{v}_{2}=0,\quad 	\breve {G}_{2j}=0 \quad \text { on } \Gamma.
\end{equation}

	The initial data are required to satisfies
\begin{equation}\label{2.11}
	\begin{aligned}
		&	\nabla^{\psi_0}\cdot(\rho_{0} G_{0,j})=0 &\text{in} ~\Omega,\\
		& h^{+}_{0}=	h^{-}_{0}, ~v^{+}_{0}\cdot n_{0}=	v^{-}_{0}\cdot n_{0},~G^{+}_{0,j} \cdot n_{0}=G^{-}_{0,j} \cdot n_{0}=0 &\text{on}~ \Gamma.
	\end{aligned}
\end{equation}

Since we are interested in Kelvin-Helmholtz instability, the instability behavior firstly happens on the boundary.  To see this,  we are going to derive an second order evolution equation for the front $f$ on the fixed boundary $\Gamma$. By using the  momentum equation of  \eqref{2.8}, we deduce that 
\begin{equation}\label{2.12}
	\begin{aligned}
		\partial^{2}_{t} f =&\partial_{t} v^{+}\cdot n + v^{+} \cdot \partial_{t} n \\
		=& -(\left(\breve{v}^{+} \cdot \nabla\right) v^{+}-\sum_{j=1}^{2} (\breve{G}^{+}_{j}\cdot \nabla ) G^{+}_{j}+ c^{2}  \nabla^{\psi} h^{ +} )\cdot n- v^{+} \cdot (\partial_{1}\partial_{t} f, 0)\\
		=& -v^{+}_{1}\partial_{1} v^{+}\cdot n+\sum_{j=1}^{2} G^{+}_{1j}\partial_{1} G^{+}_{j}\cdot n -c^{2} \nabla^{\psi} h^{+}\cdot n
		-v^{+}_{1} \partial_{1}\partial_{t} f \\
		=& v^{+}_{1}\partial_{1}  n \cdot v^{+}- v^{+}_{1}\partial_{1} \partial_{t} f+\sum_{j=1}^{2} G^{+}_{1j}\partial_{1} G^{+}_{j}\cdot n+  c^{2} \nabla^{\psi} h^{+}\cdot n
		-v^{+}_{1} \partial_{1}\partial_{t} f \\
		=& -2v^{+}_{1}\partial_{1} \partial_{t} f  -  c^{2} \nabla^{\psi} h^{+}\cdot n - (v_{1}^{+})^{2} \partial^{2}_{11} f + \sum_{j=1}^{2} (G^{+}_{1j})^{2} \partial^{2}_{11} f \quad \text{on} \quad \Gamma.
	\end{aligned}
\end{equation}
Similarly, we can also derive an evolution equation of $f$ from the negative part:
\begin{equation}\label{2.13}
	\begin{aligned}
		\partial^{2}_{t} f
		=-2v^{-}_{1}\partial_{1} \partial_{t} f  -  c^{2} \nabla^{\psi} h^{-}\cdot n - (v_{1}^{-})^{2} \partial^{2}_{11} f + \sum_{j=1}^{2} (G^{-}_{1j})^{2} \partial^{2}_{11}  f ~~~on~ \Gamma.
	\end{aligned}
\end{equation}
Therefore summing up the $``+"$ equation \eqref{2.12} and $``-"$ equation \eqref{2.13} to get
\begin{equation}\label{2.14}
	\begin{aligned}
		&\partial^{2}_{t} f+ (v^{+}_{1}+v^{-}_{1})\partial_{1} \partial_{t} f +\frac{1}{2}((c^{+})^{2} \nabla^{\psi} h^{+}\cdot n+(c^{-})^{2} \nabla^{\psi} h^{-}\cdot n)\\
		& + \frac{1}{2}((v_{1}^{+})^{2}+(v_{1}^{-})^{2}) \partial^{2}_{11} f -\frac{1}{2} \sum_{j=1}^{2} ((G^{+}_{1j})^{2}+(G^{-}_{1j})^{2}) \partial^{2}_{11}  f=0 ~~~on~ \Gamma.
	\end{aligned}
\end{equation}

\subsection{The wave equation for $h$}
\quad

Applying the equation $\partial_{t} +\breve{v} \cdot \nabla$ to the first equation of \eqref{2.8} and $\nabla^{\psi} \cdot$ to the second one of \eqref{2.8} gives
\begin{equation}\label{2.15}
	\left\{
	\begin{aligned}
		&(\partial_{t} +\breve{v}\cdot \nabla)^{2} h+ (\partial_{t} +\breve{v}\cdot \nabla)\nabla^{\psi} \cdot v=0,\\
		&\nabla^{\psi} \cdot\left((\partial_{t} +\breve{v}\cdot \nabla)v \right)+ \nabla^{\psi} \cdot (c^{2} \nabla^{\psi} h) = \nabla^{\psi} \cdot (\sum_{j=1}^{2} (\breve{G}_{j}\cdot \nabla ) G_{j}).
	\end{aligned}
	\right.
\end{equation}

Next, we take the difference of the two equations in \eqref{2.15} to deduce a wave-type equation:
\begin{equation}\label{2.16}
	(\partial_{t} +\breve{v}\cdot \nabla)^{2} h- c^{2}\Delta^{\psi} h - \sum_{j=1}^{2} (\breve{G}_{j}\cdot \nabla ) (\nabla^{\psi} \cdot G_{j})  =\mathcal{F},
\end{equation}
where the term $\mathcal{F}=-[\partial_{t}+\breve{v}\cdot \nabla, \nabla^{\psi} ]v- \nabla^{\psi} c^{2} \cdot  \nabla^{\psi} h+\sum_{j=1}^{2} \nabla^{\psi} \breve{G}_{j}\cdot \nabla  G_{j} $ is a lower order term in the second order differential equation for $h$, here $[f,g]a=f(ga)-g(fa)$ is a commutator.

As for the term $  \sum_{j=1}^{2} (\breve{G}_{j}\cdot \nabla ) (\nabla^{\psi} \cdot G_{j}) $,  by using the divergence constraint in \eqref{2.8}, we rewrite it as follows:
\begin{equation}\label{2.17}
	\nabla^{\psi} \cdot (\varrho G_{j}) =\nabla^{\psi} \varrho \cdot G_{j} + \varrho \nabla^{\psi} \cdot G_{j}=0,
\end{equation}
thus, we have
\begin{equation}\label{2.18}
	\begin{aligned}
		&\sum_{j=1}^{2} (\breve{G}_{j}\cdot \nabla ) (\nabla^{\psi} \cdot G_{j}) = 
		-\sum_{j=1}^{2} (\breve{G}_{j}\cdot \nabla ) (\frac{\nabla^{\psi} \varrho}{\varrho}\cdot  G_{j}) = -\sum_{j=1}^{2} (\breve{G}_{j}\cdot \nabla ) (\nabla^{\psi} h\cdot  G_{j}) \\
		&= -\sum_{j=1}^{2}  G_{sj} \partial^{\psi}_{s} (\partial^{\psi}_{l} h G_{lj})= -\sum_{j=1}^{2}  G_{sj} G_{lj}  \partial^{\psi}_{s} \partial^{\psi}_{j} h- G_{sj} \partial^{\psi}_{s}  G_{lj} \partial^{\psi}_{l} h.
	\end{aligned}
\end{equation}

Finally, substituting \eqref{2.18}  into \eqref{2.16} and we deduce a wave-type equation for $h$:
\begin{equation}\label{2.19}
	(\partial_{t} +\breve{v}\cdot \nabla)^{2} h- c^{2}\Delta^{\psi} h -\sum_{j=1}^{2}  G_{sj} G_{lj}  \partial^{\psi}_{s} \partial^{\psi}_{j} h=G_{sj} \partial^{\psi}_{s}  G_{lj} \partial^{\psi}_{l} h+ \mathcal{F},
\end{equation}
note that all terms in the right hand side of \eqref{2.19} are lower order terms.

	From the boundary conditions in \eqref{2.8}, we already know that 
\begin{equation}\label{2.20}
	[h]=0~~\text{on}~\Gamma.
\end{equation}

To determine the value of $h$, we add another condition involving the normal derivatives of $h$. More precisely, Taking  the difference of  two equations \eqref{2.12} and  \eqref{2.13}, we can  obtain the jump of the normal derivatives $\nabla h^{\pm}\cdot n$ on $\Gamma$,
\begin{equation}\label{2.21}
	[c^{2} \nabla^{\psi} h\cdot n]=[-2v_{1}\partial_{1}\partial_{t} f- (v_{1})^{2} \partial^{2}_{11} f+\sum_{j=1}^{2} (G_{1j})^{2} \partial^{2}_{11}  f]~~on~\Gamma.
\end{equation}

Combining \eqref{2.19}, \eqref{2.20} with \eqref{2.21} to obtain the nonlinear system for  $h$,
\begin{equation}\label{2.22}
	\left\{
	\begin{aligned}
		&	(\partial_{t} +\breve{v}\cdot \nabla)^{2} h- c^{2}\Delta^{\psi} h -\sum_{j=1}^{2}  G_{sj} G_{lj}  \partial^{\psi}_{s} \partial^{\psi}_{j} h=G_{sj} \partial^{\psi}_{s}  G_{lj} \partial^{\psi}_{l} h+ \mathcal{F}&~~on~\Omega, \\
		&[h]=0&~~on~\Gamma,\\
		&		[c^{2} \nabla^{\psi} h\cdot n]=[-2v_{1}\partial_{1}\partial_{t} f- (v_{1})^{2} \partial^{2}_{11} f+\sum_{j=1}^{2} (G_{1j})^{2} \partial^{2}_{11}  f]&~~on~\Gamma.
	\end{aligned}
	\right.
\end{equation}

\subsection{Definitions and Terminology}
\qquad	Before stating the main result, we define some notation that will be used throughout the paper. Throughout
the paper, $C>0$ will denote a generic constant that can depend on the parameters
of the problem, but does not depend on the data, etc. We refer to such 
constants as universal constants. They are allowed to change from one inequality to the
next one. We will employ the notation $a\lesssim b$ to mean that $a\leq C b$  for a universal
constant $C>0$. Also the notation $a\gtrsim b$ denotes $a\geq C b$. $\mathfrak{R} $ means the real part of the complex function or number. We use $x$ to mean the 2 dimensional vector $(x_1, x_2)$. It is conventional that $e_2=(0,1)$ means the unit vector in $\mathbb{R}^2$.

Since we study two disjoint fluids, for a function $\psi$ defined in $\Omega$, we write $\psi^{+}$ for the restriction to $\Omega^{+}$ and $\psi^{-}$ for the restriction to $\Omega^{-}$. For all $j\in \mathbb{R}$, We  define the piecewise Sobolev space by 
\begin{equation}\label{2.23}
	H^{j}(\Omega): =\{  \psi| \psi^{+} \in H^{j}(\Omega^{+}), \psi^{-} \in H^{j}(\Omega^{-}) \},
\end{equation}
endowed with the norm $\|\psi\|_{H^{j}}^{2}= \|\psi^{+}\|_{H^{j}(\Omega^{+})}^{2}+\|\psi^{-}\|_{H^{j}(\Omega^{-})}^{2}$.
The usual Sobolev norm $ \|\psi\|_{H^{j}(\Omega^{\pm})}^{2}$ is equipped with the following norm:
\begin{equation}\label{2.24}
	\begin{aligned}
		\|\psi\|_{H^{j}(\Omega^{\pm})}^{2}: &=\sum_{s=0}^{j} \int_{\mathbb{R}\times I_{\pm}} (1+\eta^{2})^{j-s} |\partial_{2}^{s}\hat{\psi}_{\pm}(\eta,x_{2})|^{2} d \eta d x_{2}\\
		&= \sum_{s=0}^{j} \int_{\mathbb{R}} (1+\eta^{2})^{j-s} \|\partial_{2}^{s}\hat{\psi}_{\pm}(\eta,x_{2})\|^{2}_{L^{2}(I_{\pm})}  d \eta,
	\end{aligned}
\end{equation}
where $I_{-}=(-\infty,0)$ and $I_{+}=(0,\infty)$ and $\hat{\psi}$ is the Fourier transform of $ \psi $ via
\begin{equation}\label{2.25}
	\hat{\psi}(\eta)=  \int_{\mathbb{R}} \psi e^{-i x_{1} \eta}  dx_{1}, 
\end{equation}
for a function $\psi$ defined on $\Gamma$, we define usual Sobolev space by 
\begin{equation}\label{2.26}
	\|\psi\|_{H^{j}(\Gamma)}^{2}: =  \int_{\mathbb{R}} (1+\eta^{2})^{j} |\hat{\psi}(\eta)|^{2} d \eta.
\end{equation}
To shorten notation, for $k \geq 0$, we define
\begin{equation}\label{2.27}
	\|(f, h, v,G_{j})(t)\|_{H^{k}}=\|f(t)\|_{H^{k}(\Gamma)}+\|h(t)\|_{H^{k}(\Omega)}+\|v(t)\|_{H^{k}(\Omega)}+\|G_{j}(t)\|_{H^{k}(\Omega)}. 
\end{equation}

\subsection{Main result}
\quad  This paper is devoted to proving the ill-posedness of  Kelvin-Helmholtz problem of the ideal compressible elastic system under  the following  condition:
\begin{equation}\label{2.28}
	U_{low}+c\epsilon_{0}\le |\dot{v}^{+}_{1}|\le U_{upp}-c\epsilon_{0},
\end{equation}
where we define $U_{low}:=\sqrt{((\dot{G}^{+}_{11})^{2}+(\dot{G}^{+}_{12})^{2})}$ and $U_{upp}:=\sqrt{2c^{2}+ ((\dot{G}^{+}_{11})^{2}+(\dot{G}^{+}_{12})^{2})}$, here $\epsilon_{0}$ is a small but fixed constant, $\dot{G}^{+}_{11} $ and $\dot{G}^{+}_{12}$ are defined by \eqref{3.3}, $c=c(\dot{\rho})$ is the sound speed defined as \eqref{1.2}.

\begin{defi}\label{definition}
	We say that the  problem \eqref{2.8} is locally well-posedness for some $k \geq 3$ if there exist $\delta, t_{0}, C>0$  such that for any initial data $(f_{0}^{1}, h_{0}^{1}, v_{0}^{1},G_{0,j}^{1})$, $(f_{0}^{2}, h_{0}^{2}, v_{0}^{2},G_{0,j}^{2})$ satisfying
	\begin{equation}\label{2.29}
		\|(f_{0}^{1}-f_{0}^{2}, h_{0}^{1}-h_{0}^{2}, v_{0}^{1}-v_{0}^{2},G_{0,j}^{1}-G_{0,j}^{2})\|_{H^{k}}<\delta,
	\end{equation}
	there exist unique solutions $(f^{1}, h^{1}, v^{1},G_{j}^{1})$ and  $(f^{2}, h^{2}, v^{2},G_{j}^{2})\in L^{\infty}([0,t_{0}]; H^{3})$ for the system \eqref{2.8} with initial data $(f^{k}, h^{k}, v^{k},G^{k}_{j})|_{t=0}=(f^{k}_{0}, h^{k}_{0}, v^{k}_{0},G^{k}_{0,j})$, $k=1,2$ and there holds
	\begin{equation}\label{2.30}
		\begin{aligned}
			&\sup_{0\leq t \leq t_{0}}	\|(f^{1}-f^{2}, h^{1}-h^{2}, v^{1}-v^{2},G^{1}_{j}-G^{2}_{j})(t)\|_{H^{3}}\\
			&\leq C(\| (f_{0}^{1}-f_{0}^{2}, h_{0}^{1}-h_{0}^{2}, v_{0}^{1}-v_{0}^{2},G^{1}_{0,j}-G^{2}_{0,j}) \|_{H^{k}}).
		\end{aligned}
	\end{equation}
\end{defi}

\begin{theo}\label{theorem:main}
	Let the initial domain to be $\mathbb{R}^{2}= \Omega^{+}(0) \cup \Omega^{-}(0) \cup \Gamma(0)$. Suppose that the initial data satisfies the constraint  condition \eqref{2.3} and \eqref{2.11}, further we assume  the  rectilinear solution  satisfies the instability condition \eqref{2.28}.  Then the Kelvin-Helmholtz problem of \eqref{2.8}  is not locally well-posed in the sense of Definition \ref{definition}.
\end{theo}

\begin{rema}
	We construct the growing normal mode solution for the front $f$ when the background velocity satisfies $U_{low}<|\dot{v}^{+}_{1}|<U_{upp}$.  While for the linear and nonlinear problem, we can prove the ill-posedness of the solutions $(h, v, G_{j})$ of the Kelvin-Helmholtz problem to the ideal compressible elastic flow uniformly when $U_{low}+c\epsilon_{0}\le \dot{v}^{+}_{1}\le U_{upp}-c\epsilon_{0}$, where $\epsilon_0$ is some fixed small enough positive constant and $c$ is the sound speed. 
\end{rema}
\begin{rema}
	Since $\Psi: (t,x)\mapsto (t,x_{1}, x_{2}+ \psi(t,x)) $ is a diffeomorphism transform, the ill-poseness of system \eqref{2.8} in the flatten coordinate implies the ill-poseness of the solution to the original system \eqref{1.1}-\eqref{1.6}.
\end{rema}

\begin{rema}
	In this paper, we take $(f^{2}, h^{2}, v^{2},G_{j}^{2})$ exactly by the rectilinear solution $(\dot{f},\dot{\varrho}^{\pm},\dot{v}^{\pm},\dot{G}_{j}^{\pm})$.  Thus we prove the ill-posedness of the solution near the rectilinear solution for the free boundary value problem of the two-dimensional ideal 
	compressible elastic flows.
\end{rema}

\begin{rema}
	By taking $\dot{G}=0$, the  instability condition \eqref{2.28} becomes 
	\begin{equation}\label{2.31}
		\epsilon_{0} \leq M:=\frac{|\dot{v}^{+}_{1}|} {c}\le \sqrt{2}-\epsilon_{0},
	\end{equation}
 which reduces to the instability condition  of Kelvin-Helmholtz problem  for  the compressible  Euler fluids(refer to \cite{Xie1} ) .
\end{rema}

	\section{The Linearized Equations in new coordinate}
\quad \quad In this section, we consider  a  linearized system in new coordinate and would construct a  growing normal mode solution for this linearized system. By taking Fourier trnasform of linearized system and solving the linear equation for $\hat{m}$, we finally derive a second order ordinary equation for $\hat{g}$ and obtain the symbol for $\hat{g}$. With the analysis of the symbol of $\hat{g}$, we show the instability condition of the linear system of Kelvin-Helmholtz problem  for  the ideal compressible elastic fluids.

\subsection{Construction of a growing solution of the  linearized system.}
\quad \quad  It is easily verified that the particular solution in Euler coordinate is also a  particular solution in new coordinate such that
\begin{equation} \label{3.1}
	\dot{v}^{\pm}=\dot{u}^{\pm}=\left\{
	\begin{aligned}
		&(\dot{v}^{+}_{1},0)&x_2\ge 0,\\
		&(\dot{v}^{-}_{1},0)&x_2<0,
	\end{aligned}
	\right.
\end{equation}
and 
\begin{equation}\label{3.2}
	\dot{\varrho}^{+}=\dot{\varrho}^{-}:=\dot{\varrho}.
\end{equation}
and
\begin{equation} \label{3.3}
\dot{G}	=\dot{F}=\left\{
	\begin{aligned}
		&\left[ {\begin{array}{cc}\dot{G}^{+}_{11} & \dot{G}^{+}_{12} \\0 & 0 \\\end{array} } \right]&x_2\geq0,\\
		&\left[ {\begin{array}{cc}\dot{G}^{-}_{11} & \dot{G}^{-}_{12} \\0 & 0 \\\end{array} } \right]&x_2<0.
	\end{aligned}
	\right.
\end{equation}

Now we will consider a constant coefficient linearized equations which is derived by linearization of equations \eqref{2.8}, \eqref{2.14} and \eqref{2.22} around the rectilinear solutions: constant velocity $\dot{v}^{\pm}=(\dot{v}^{\pm}_{1},0)$, constant deformation matrix $\dot{G}^{\pm}=(\dot{G}^{\pm}_{11}, \dot{G}^{\pm}_{12},0,0)$, constant pressure $\dot{h}^{+}= \dot{h}^{-}$,  flat front $\Gamma=\{x_{2}=0 \}$ and around the outer normal vector $e_2=(0,1)$. Moreover, all the rectilinear solutions can be transformed to the following form  under the Galilean transformation and by the change of the scale of measurement
\begin{equation}\label{3.4}
	\dot{v}^{+}_{1}+ \dot{v}^{-}_{1}=0, ~\dot{G}^{+}_{11}+ \dot{G}^{-}_{11}=0,~\dot{G}^{+}_{12}+ \dot{G}^{-}_{12}=0.
\end{equation}
\begin{rema}
From now on and throughout this paper, We use the new notation $(\dot{f},\dot{\varrho}^{\pm},\dot{v}^{\pm},\dot{G}^{\pm})$ to denoe the rectilinear solution  $(\dot{f},\dot{\rho}^{\pm},\dot{u}^{\pm},\dot{F}^{\pm})$, which is in fact the same constant quantity. Here we use the new notation to match the notation in the new coordinates.
\end{rema}

Therefore, we have the following linearized equation around the rectilinear solution  $(\dot{f},\dot{\varrho}^{\pm},\dot{v}^{\pm},\dot{G}^{\pm})$
	\begin{equation}\label{3.5}
	\begin{cases}
		\partial_{t} h+\dot{v}_{1} \partial_{1} h+ {\rm div} v=0 & \text { in } \Omega, \\
		\partial_{t} v+\dot{v}_{1} \partial_{1} v+ c^{2}(\dot{h}) \nabla h=\sum_{j=1}^{2} \dot{G}^{+}_{1j}\partial_{1} G^{+}_{j} & \text { in } \Omega, \\
		\partial_{t} G_{j}+\dot{v}_{1} \partial_{1}  G_{j} =\dot{G}_{1j}\partial_{1} v & \text { in } \Omega, \\
		\partial_{t} f=v_{2}-\dot{v}_{1} \partial_{1} f   & \text { on } \Gamma.
	\end{cases}
\end{equation}
To linearize the boundary conditions in  \eqref{2.8}, we let $v=\dot{v}+ \tilde{v}$ and $n=e_{2}+ \tilde{n}$, therefore we have:
\begin{equation*}
	[(\dot{v}+ \tilde{v})\cdot (e_{2}+ \tilde{n}) ]=[\tilde{v} \cdot e_{2}]+ [\dot{v}\cdot \tilde{n} ]+ [\tilde{v} \cdot \tilde{n}]=0,
\end{equation*}
where $\tilde{n}=(-\partial_{1}\tilde{f},0 )$. Obviously, the third term is nonlinear term, it follows that 
\begin{equation*}
	[\tilde{v} \cdot e_{2}] =- [\dot{v}\cdot \tilde{n} ]= 2\dot{v}^{+}_{1}\partial_{1} \tilde{f}.
\end{equation*}
Similarly, we can  deduce that 
\begin{equation*}
	 \tilde{G}_{2j} - \dot{G}_{1j}\partial_{1} \tilde{f}=0.
\end{equation*}
Thus, the boundary conditions can be  linearized as follows:
\begin{equation}\label{3.6}
	[h]=0,~[v\cdot e_{2}]= 2\dot{v}^{+}_{1}\partial_{1} f,~ G_{2j} - \dot{G}_{1j}\partial_{1} f=0 \text { on } \Gamma.
\end{equation}

We also get a linearized equation for the front $f$
\begin{equation}\label{3.7}
	\begin{aligned}
		&\partial_{t}^{2}f+ (\dot{v}^{+}_{1})^{2} \partial^{2}_{11} f+ \frac{c^{2}}{2} \partial_{2} (h^{+}+h^{-})- ((\dot{G}^{+}_{11})^{2}+(\dot{G}^{+}_{12})^{2})\partial^{2}_{11} f=0 ~~~on~\Gamma,
	\end{aligned}
\end{equation}
	and a linearized system for  the pressure $h$
\begin{equation}\label{3.8}
	\left\{
	\begin{aligned}
		&(\partial_{t} +\dot{v}_{1}\partial_{1})^{2} h-  c^{2} \Delta h - ((\dot{G}^{+}_{11})^{2}+(\dot{G}^{+}_{12})^{2})\partial^{2}_{11} h =0, &~~~on~ \Omega,\\
		&[h]=0, &~~~on~ \Gamma,\\
		&[c^{2}\partial_{2} h]= -4\dot{v}^{+}_{1} \partial_{t}\partial_{1}f& ~~~on~ \Gamma.
	\end{aligned}
	\right.
\end{equation}

Since we want to construct a solution to the linear system \eqref{3.5}-\eqref{3.8} that has a growing $H^{k}$ norm for any $k$. To begin with, we assume the solution is in the following normal mode form:
\begin{equation}\label{3.9}
		\begin{aligned}
	&h(t,x_{1},x_{2})= e^{\tau t} m(x_{1},x_{2}),~ v(t,x_{1},x_{2})= e^{\tau t} w(x_{1},x_{2}),\\ &G_{j}(t,x_{1},x_{2})= e^{\tau t} E_{j}(x_{1},x_{2}),~ f(t,x_{1})= e^{\tau t} g(x_{1}),
		\end{aligned}
\end{equation}
here we assume that $\tau=\gamma+ i \delta \in \mathbb{C}\backslash \{0\}  $ is the same above and below the interface. A solution with $\mathfrak{R} (\tau)>0$ corresponds to a growing mode. Plugging the ansatz \eqref{3.8} into  \eqref{3.5}-\eqref{3.8}, we have 
\begin{equation}\label{3.10}
	\begin{cases}
		\tau  m+\dot{v}_{1} \partial_{1} m+ {\rm div} w=0 & \text { in } \Omega, \\
		\tau w+\dot{v}_{1} \partial_{1} w+ c^{2} \nabla m=\sum_{j=1}^{2} \dot{G}^{+}_{1j}\partial_{1} E^{+}_{j} & \text { in } \Omega, \\
		\tau E_{j}+\dot{v}_{1} \partial_{1}  E_{j} =\dot{G}_{1j}\partial_{1} w & \text { in } \Omega, \\
		\tau g=w_{2}-\dot{v}_{1} \partial_{1} g & \text { on } \Gamma, \\
		  [m]=0,~[w\cdot e_{2}]= 2\dot{v}^{+}_{1}\partial_{1} g,~ E_{2j} - \dot{G}_{1j}\partial_{1} g=0 & \text { on } \Gamma.
	\end{cases}
\end{equation}
and
\begin{equation}\label{3.11}
	\begin{aligned}
		&\tau^{2}g + (\dot{v}^{+}_{1})^{2} \partial^{2}_{11} g+ \frac{c^{2}}{2} \partial_{2} (m^{+}+m^{-})- ((\dot{G}^{+}_{11})^{2}+(\dot{G}^{+}_{12})^{2})\partial^{2}_{11} g=0 ~~~on~\Gamma,
	\end{aligned}
\end{equation}
and 
\begin{equation}\label{3.12}
	\left\{
	\begin{aligned}
		&(\tau +\dot{v}_{1}\partial_{1})^{2} m-  c^{2} \Delta m - ((\dot{G}^{+}_{11})^{2}+(\dot{G}^{+}_{12})^{2})\partial^{2}_{11} m =0 &~~~on~ \Omega,\\
		&[m]=0 &~~~on~ \Gamma,\\
		&[c^{2}\partial_{2} m]= -4\dot{v}^{+}_{1} \tau\partial_{1}g& ~~~on~ \Gamma.
	\end{aligned}
	\right.
\end{equation}

\subsection{The  formula  for $\partial_{2}\hat{m}^{+}+\partial_{2}\hat{m}^{-}$ on $\Gamma$}
\quad By taking the Fouier transform of problem \eqref{3.11} and \eqref{3.12}, we deduce a formula for $\partial_{2}\hat{m}^{+}+\partial_{2}\hat{m}^{-}$ on $\Gamma$, then substituting  this formula into \eqref{3.11}, it follows that an second-order wave-type equation  for the front  $g$.  More precisely, we define the Fourier  transform of $m$ and $g$ with respect to the $x_1$ variable as follow:
\begin{equation*}
\hat{m}(\eta,x_{2})= \int_{\mathbb{R}} m(x_{1},x_{2}) e^{-ix_{1} \eta}  dx_{1} ,~ \hat{g}(\eta)= \int_{\mathbb{R}} g(x_{1}) e^{-i x_{1} \eta} dx_{1}.
\end{equation*}

Taking the transform of the problem \eqref{3.11} and \eqref{3.12}, we derive the following equations
\begin{equation}\label{3.13}
	\begin{aligned}
		\tau^{2}\hat{g}-(\dot{v}^{+}_{1})^{2} \eta^{2} \hat{g}+ \frac{c^{2}}{2} \partial_{2} (\hat{m}^{+}+\hat{m}^{-})+ ((\dot{G}^{+}_{11})^{2}+(\dot{G}^{+}_{12})^{2}) \eta^{2} \hat{g}=0~~\text{on}~\Gamma,
	\end{aligned}
\end{equation}
and
\begin{equation}\label{3.14}
	\left\{
	\begin{aligned}
		&(\tau +i\dot{v}_{1} \eta)^{2} \hat{m}+c^{2}\eta^{2} \hat{m}-  c^{2} \partial^{2}_{22} \hat{m}+ ((\dot{G}^{+}_{11})^{2}+(\dot{G}^{+}_{12})^{2})\eta^{2} \hat{m}=0~~~&\text {in}~ \Omega,\\
		&[\hat{m}]=0~~~&\text{on}~\Gamma,\\
		&[c^{2}\partial_{2} \hat{m}]= -4 i\dot{v}^{+}_{1} \tau\eta  \hat{g}~~~&\text{on}~\Gamma.
	\end{aligned}
	\right.
\end{equation}

Solving the system  \eqref{3.14}, we obtain 
\begin{equation} \label{3.15}
	\hat{m}=\left\{
	\begin{aligned}
		& \frac{4 i\dot{v}^{+}_{1}\tau \eta  \hat{g}}{c^{2}(\mu^{+}+\mu^{-})}e^{-\mu^{+}x_2}&x_2\ge 0,\\
		& \frac{4 i \dot{v}^{+}_{1} \tau \eta \hat{g}}{c^{2}(\mu^{+}+\mu^{-})}e^{\mu^{-}x_2}&x_2<0,
	\end{aligned}
	\right.
\end{equation}
where $\mu^{\pm}=\sqrt{\frac{(\tau\pm i\dot{v}^{+}_{1} \eta)^{2} + ((\dot{G}^{+}_{11})^{2}+(\dot{G}^{+}_{12})^{2}) \eta^{2}}{c^{2} }+ \eta^{2}}$ are the root of the equation 
\begin{equation}\label{3.16}
	-c^{2} s^{2} +(\tau\pm i \dot{v}^{+}_{1}\eta)^{2}+ ((\dot{G}^{+}_{11})^{2}+(\dot{G}^{+}_{12})^{2}) \eta^{2}+c^{2}\eta^{2}=0,
\end{equation}
here we notice that $\mathcal{R} \mu^{\pm} >0$ since $\mathcal{R} \tau >0$.

By direct computation, we deduce
\begin{equation}\label{3.17}
	c^{2}(\partial_{2}\hat{m}^{+}+\partial_{2}\hat{m}^{-})=-4i\dot{v}^{+}_{1} \eta \tau\hat{g} \frac{\mu^{+}-\mu^{-}}{\mu^{+}+\mu^{-}}~~ ~~\text{on}~\Gamma,
\end{equation}
substituting \eqref{3.17} into \eqref{3.13} to obtain an second-order  equation for $\hat{g}$
\begin{equation}\label{3.18}
	\begin{aligned}
		(\tau^{2}-(\dot{v}^{+}_{1})^{2} \eta^{2}-2i \dot{v}^{+}_{1} \eta \tau \frac{\mu^{+}-\mu^{-}}{\mu^{+}+\mu^{-}} + ((\dot{G}^{+}_{11})^{2}+(\dot{G}^{+}_{12})^{2})  \eta^{2}) \hat{g}=0~~~~\text{on}~\Gamma.
	\end{aligned}
\end{equation}

\subsection{Study of the roots of the symbol $\Sigma$ }
\quad In order to further analysis  the equation \eqref{3.18}, we  define  the symbol of \eqref{3.18} by $\Sigma$:
\begin{equation}\label{3.19}
	\Sigma(\tau,\eta)= \tau^{2}-(\dot{v}^{+}_{1})^{2} \eta^{2} -2i \dot{v}^{+}_{1} \eta \tau \frac{\mu^{+}-\mu^{-}}{\mu^{+}+\mu^{-}}+((\dot{G}^{+}_{11})^{2}+(\dot{G}^{+}_{12})^{2}) \eta^{2},
\end{equation}
and define a set of ``frequencies"
\begin{equation}\label{3.20}
	\Xi= \{ (\tau,\eta)\in \mathbb{C} \times \mathbb{R}  : \mathfrak{R} \tau> 0, (\tau,\eta)\neq (0,0)\}.
\end{equation}

Since we already know that $\mathfrak{R} \mu^{\pm}>0$ in all points with $\mathfrak{R} \tau>0$. It follows that $\mathfrak{R} (\mu^{+}+ \mu^{-})>0$ and thus $\mu^{+}+ \mu^{-}>0$ in all such points. In this paper, we are interested only in the unstable case, the symbol $\Sigma$ is defined in points $(\tau,\eta)\in \Xi$.

	We also need to know whether the difference $\mu^{+}-\mu^{-}$ vanishes.
\begin{lemm}
	Let $(\tau,\eta)\in \Xi$. Then $\mu^{+}=\mu^{-}$ if and only if $(\tau,\eta)=(\tau,0)$.
\end{lemm}
\begin{proof}
	From \eqref{3.16}, it implies that  $(\mu^{+})^{2}=(\mu^{-})^{2}$ if and only if $\eta=0$ or $\tau=0$. Since  $(\tau,\eta)\in \Xi$, only  $\eta=0$ case need to study. When $\eta=0$, it follows that $\mu^{+}=\mu^{-}= \tau/c$.
\end{proof}

We will discuss the roots of  the symbol \eqref{3.19} in the instability case.
\begin{lemm}\label{lem:symbol}
	Let $\Sigma(\tau,\eta)$ be the symbol defined in \eqref{3.19}, for $(\tau, \eta) \in \Xi$. If   $U_{low}<\dot{v}^{+}_{1}<U_{upp}$, then $\Sigma(\tau,\eta)=0$ if only if
	\begin{equation}\label{3.21}
		\tau=  X_{1} \eta,
	\end{equation}
	where $X_{1}^{2}=\sqrt{c^{4}+ 4((\dot{G}^{+}_{11})^{2}+(\dot{G}^{+}_{12})^{2}+c^{2})(\dot{v}^{+}_{1})^{2}}- (\dot{v}^{+}_{1})^{2}-(\dot{G}^{+}_{11})^{2}-(\dot{G}^{+}_{12})^{2}-c^{2}>0$. The root $	\tau=  X_{1} \eta$ is simple, i.e. there exists a neighborhood $\mathcal{V}$ of $(X_{1}\eta, \eta)\in \Xi$ and a smooth $F$ defined on  $\mathcal{V}$ such that 
	\begin{equation*}
		\Sigma= (\tau-X_{1}\eta) F(\tau,\eta),~F(\tau,\eta)\neq 0 ~\text{for all} (\tau,\eta)\in \mathcal{V}, 
	\end{equation*}
	where $F(\tau, \eta)$ is defined as $c^{2}\eta^{2} \frac{d\phi }{d X}(\alpha X_1+ (1-\alpha) X))$.
\end{lemm}

\begin{proof}
	
	In according with the definition of $\Sigma$ and Lemma 3.1,  we can easily  verify $\Sigma(\tau,0)=\tau^{2}\neq 0$ for $(\tau,0)\in \Xi$.  Meanwhile, it is easy to check that $\Sigma(\tau, \eta)= \Sigma(\tau, -\eta)$. Thus we can assume without loss of generality that $\tau\neq 0$, $\eta\neq 0 $ and $ \eta>0$ and from Lemma 3.1 we know that $\mu^{+}- \mu^{-}\neq 0$. Therefore we compute
	\begin{equation}\label{3.22}
		\begin{aligned}
			\frac{\mu^{+}-\mu^{-}}{\mu^{+}+\mu^{-}}= \frac{(\mu^{+}-\mu^{-})^{2}}{(\mu^{+})^{2}-(\mu^{-})^{2}}= \frac{c^{2}(\mu^{+}-\mu^{-})^{2}}{4i \dot{v}^{+}_{1}\tau},
		\end{aligned}
	\end{equation}
	and
	\begin{equation}\label{3.23}
		\begin{aligned}
			(\mu^{+}-\mu^{-})^{2}= 2((\frac{\tau}{c})^{2}- (\frac{\dot{v}^{+}_{1}\eta}{c})^{2}+ \frac{((\dot{G}_{11}^{+})^{2}+(\dot{G}_{12})^{2})\eta^{2}}{c^{2}} +\eta^{2}- \mu^{+}\mu^{-}),
		\end{aligned}
	\end{equation}
	therefore we deduce that
	\begin{equation}\label{3.24}
		\begin{aligned}
			\frac{\mu^{+}-\mu^{-}}{\mu^{+}+\mu^{-}}=\frac{2[{\tau}^{2}- (\dot{v}^{+}_{1}\eta)^{2}+ ((\dot{G}_{11}^{+})^{2}+(\dot{G}_{12})^{2})\eta^{2} +c^{2}(\eta^{2}- \mu^{+}\mu^{-})]}{4i \dot{v}^{+}_{1} \tau}
		\end{aligned}
	\end{equation}
	and substituting the above expression \eqref{3.24} into \eqref{3.18} we can rewrite it as
	\begin{equation}\label{3.25}
		\begin{aligned}
			c^{2}( \mu^{+}\mu^{-}-\eta^{2})\hat{g}=0,~~\text{on}~\Gamma.
		\end{aligned}
	\end{equation}
	The symbol $\Sigma$ can be  reformulated as
	\begin{equation}\label{3.26}
		\begin{aligned}
			\Sigma= c^{2}( \mu^{+}\mu^{-}-\eta^{2}).
		\end{aligned}
	\end{equation}
	
	Let us set $\mu^{+}\mu^{-}-\eta^{2}=0$ and introduce two quantities:
	\begin{equation}\label{3.27}
		X=\frac{\tau}{ \eta},~\tilde{\mu}^{\pm}= \frac{\mu^{\pm}}{ \eta}.
	\end{equation}
	Therefore we can deduce
	\begin{equation}\label{3.28}
		\tilde{\mu}^{+}\tilde{\mu}^{-}=-1,
	\end{equation}
	and
	\begin{equation}\label{3.29}
		(\tilde{\mu}^{+})^{2}(\tilde{\mu}^{-})^{2}=1.
	\end{equation}
	
	By the formula of the roots $\mu^{\pm}$, it follows that
	\begin{equation}\label{3.30}
		(\tilde{\mu}^{+})^{2}=\frac{1}{c^{2}}[(X+i\dot{v}^{+}_{1})^{2}+(\dot{G}^{+}_{11})^{2}+(\dot{G}^{+}_{12})^{2} ]+1,
	\end{equation}
	and
	\begin{equation}\label{3.31}
		(\tilde{\mu}^{-})^{2}=\frac{1}{c^{2}}[(X-i\dot{v}^{+}_{1})^{2}+(\dot{G}^{+}_{11})^{2}+(\dot{G}^{+}_{12})^{2} ]+1,
	\end{equation}
	Hence we have
	\begin{equation}\label{3.32}
		[(X+i\dot{v}^{+}_{1})^{2}+(\dot{G}^{+}_{11})^{2}+(\dot{G}^{+}_{12})^{2}+c^{2}][(X-i\dot{v}^{+}_{1})^{2}+(\dot{G}^{+}_{11})^{2}+(\dot{G}^{+}_{12})^{2}+c^{2}]=c^{4},
	\end{equation}
	which leads to the following equation for $X^{2}$:
	\begin{equation}\label{3.33}
		\begin{aligned}
			&X^{4} + 2 ((\dot{v}^{+}_{1})^{2}+(\dot{G}^{+}_{11})^{2}+(\dot{G}^{+}_{12})^{2}+c^{2}) X^{2} + (\dot{v}^{+}_{1})^{4}- 2((\dot{G}^{+}_{11})^{2}+(\dot{G}^{+}_{12})^{2}+c^{2}) (\dot{v}^{+}_{1})^{2} \\
			&+ ((\dot{G}^{+}_{11})^{2}+(\dot{G}^{+}_{12})^{2})^{2}+ 2c^{2} ((\dot{G}^{+}_{11})^{2}+(\dot{G}^{+}_{12})^{2})=0.
		\end{aligned}
	\end{equation}
	Using the quadratic formula, the two roots of the above equation are
	\begin{equation}\label{3.34}
		X_{1}^{2}= -(\dot{v}^{+}_{1})^{2}-(\dot{G}^{+}_{11})^{2}-(\dot{G}^{+}_{12})^{2}-c^{2}+\sqrt{c^{4}+ 4((\dot{G}^{+}_{11})^{2}+(\dot{G}^{+}_{12})^{2}+c^{2})(\dot{v}^{+}_{1})^{2}},
	\end{equation}
	and
	\begin{equation}\label{3.35}
		X_{2}^{2}=- (\dot{v}^{+}_{1})^{2}-(\dot{G}^{+}_{11})^{2}-(\dot{G}^{+}_{12})^{2}-c^{2}- \sqrt{c^{4}+ 4((\dot{G}^{+}_{11})^{2}+(\dot{G}^{+}_{12})^{2}+c^{2})(\dot{v}^{+}_{1})^{2}},
	\end{equation}
	We claim that the points $(\tau,\eta)\in \Sigma$ with $\tau=\pm  X_{2} \eta$ are not the roots of $\mu^{+}\mu^{-}= \eta^{2}$. Without loss of generality, we can assume that $Y_{2}$ is positive. From \eqref{3.35}, we deduce
	\begin{equation}\label{3.38}
		X_{2}=iY_{2},~Y_{2}\geq \dot{v}^{+}_{1} + \sqrt{ (\dot{G}^{+}_{11})^{2}+(\dot{G}^{+}_{12})^{2}+c^{2}},
	\end{equation}
	from this  we deduce $ Y_{2}\pm \dot{v}^{+}_{1}>\sqrt{ (\dot{G}^{+}_{11})^{2}+(\dot{G}^{+}_{12})^{2}+c^{2}}$. In accord with the equation \eqref{3.32} and \eqref{3.33}, it follows  that $$\tilde{\mu}^{+}= i \sqrt{\frac{(Y_{2}+\dot{v}^{+}_{1})^{2} - (\dot{G}^{+}_{11})^{2}-(\dot{G}^{+}_{12})^{2}}{c^{2}} -1},\tilde{\mu}^{-}=  i \sqrt{\frac{(Y_{2}-\dot{v}^{+}_{1})^{2}-(\dot{G}^{+}_{11})^{2}-(\dot{G}^{+}_{12})^{2}}{c^{2}} -1}.$$
	From which we know that   $\tilde{\mu}^{+}\tilde{\mu}^{-}=1$  is not satisfied. Similarly, we can show that $(\tau,\eta)\in \Sigma$ with $\tau=- X_{2}\eta$ is not root of $\mu^{+}\mu^{-}= \eta^{2}$. On the other hand, from \eqref{3.38}, we know that $\tau=iY_{2}\eta$ is imaginary root, thus it implies that $\mathfrak{R} \tau=0$ and  $(\pm X_{2}\eta, \eta)\nsubseteq \Xi$.

	Now we  focus on  the root $X_{1}^{2}$. If $U_{low}<\dot{v}^{+}_{1}<U_{upp}$, from \eqref{3.34}, we know that $X_{1}^{2}$ is positive, it follows that $\tau = \pm  X_{1} \eta$ are real. The point $(-  X_{1} \eta,\eta)\nsubseteq \Xi$, thus we  omit this point, we only study the root $\tau = +  X_{1} \eta$. Using a fact that square roots of the complex number $a+ib$ are 
	\begin{equation}\label{3.39}
		\pm\{\sqrt{\frac{r+a}{2}} + i sgn(b)\sqrt{\frac{r-a}{2}}\},~r=|a+ib|.
	\end{equation}
	In our case here, we compute
	\begin{equation}\label{3.40}
		\mu^{+}=\sqrt{\frac{r+a}{2}} + i\sqrt{\frac{r-a}{2}},\mu^{-}=\sqrt{\frac{r+a}{2}} - i\sqrt{\frac{r-a}{2}},
	\end{equation}
	where 
	\begin{equation}\label{3.41}
		a=\frac{X^{2}_{1}- (\dot{v}^{\pm}_{1})^{2}+(\dot{G}^{+}_{11})^{2}+(\dot{G}^{+}_{12})^{2}+ c^{2}}{c^{2}} \eta^{2},~b= \frac{2X_{1}\dot{v}^{\pm}}{c^{2}} \eta^{2},
	\end{equation}
	so that $	\mu^{+}	\mu^{-}=r>0$, therefore we deduce that in case of $U_{low}<\dot{v}^{+}_{1}<U_{upp}$,    the  root of  the symbol $\Sigma$ is  the point $(+ X_{1}\eta, \eta)$. In summary we can get a root $(\tau,\eta)$ with $\mathfrak{R} \tau>0$, which is a unstable solution. 
	
	Now we prove that the root $( X_{1}\eta,\eta)$ are simple. We define $\phi(X) = \tilde{\mu}^{+} \tilde{\mu}^{-}-1$, therefore we have $\Sigma=c^{2}\eta^{2}\phi(X) $. By Taylor formula, we can write 
	\begin{equation}\label{3.42}
		\Sigma= c^{2}\eta^{2} (\phi(X_{1})+ (X-X_{1}) \frac{d\phi }{d X}(\alpha X_1+ (1-\alpha) X)), 0<\alpha<1.
	\end{equation}
	By direct computation, we have 
	\begin{equation*}
		\phi(X_{1})= 0,~ \frac{d\phi }{d X}= \frac{2X/c}{\tilde{\mu}^{+} \tilde{\mu}^{-}} \{ (\frac{X}{c})^{2}+ (\dot{v}^{+}_{1}/c)^{2}+1\}.
	\end{equation*}
	Since $\frac{d\phi }{d X} (X_{1})\neq  0$, by the continuity of  $\frac{d\phi }{d X}$, it follows that $\frac{d\phi }{d X} (\alpha X_1+ (1-\alpha) X))\neq  0$. Therefore we complete the proof of this lemma.
\end{proof}

	\section{Ill-posedness of  solutions for the linear problem}
\subsection{Uniqueness for the linearized equations \eqref{3.5}-\eqref{3.6}}
\quad To begin with, we prove a uniqueness result for the linearized system \eqref{3.5}-\eqref{3.6}. 
\begin{lemm}
	Let $f,h,v,G_{j}$ be the solution to the linearized equations \eqref{3.5}-\eqref{3.6} with the initial data $(f,h,v,G_{j})|_{t=0}=0$. Then it holds $(f,h,v,G_{j})\equiv0$.
\end{lemm}
\begin{proof}
	Taking  the standard inner product of the first equation and second equation in \eqref{3.5} with $h^{+}, v^{+}$ and integrating over $\Omega^{+}$, we obtain
	\begin{equation}\label{4.1}
		\frac{1}{2}	\partial_{t} \int_{\Omega^{+}}c^{2} |h^{+}|^{2}+\frac{1}{2} \int_{\Omega^{+}}\dot{v}_{1} \partial_{1}(c^{2} |h^{+}|^{2}) + \int_{\Omega^{+}} c^{2} h^{+} {\rm div} v^{+}=0  .
	\end{equation}
	and
	\begin{equation}\label{4.2}
		\frac{1}{2}	\partial_{t}  \int_{\Omega^{+}}|v^{+}|^{2}+\frac{1}{2} \int_{\Omega^{+}}\bar{v}^{+}_{1} \partial_{1}(|v^{+}|^{2}) + \int_{\Omega^{+}} c^{2} \nabla h^{+} \cdot v^{+}=\sum_{j=1}^{2} \int_{\Omega^{+}} \dot{G}^{+}_{1j}\partial_{1} G^{+}_{j} v^{+}. 
	\end{equation}
	Meanwhile, we take the inner product of the third equation in \eqref{3.5} with $G_{j}$ and integrate over $\Omega^{+}$ to arrive at 
		\begin{equation}\label{4.3}
		\frac{1}{2}	\partial_{t}  \int_{\Omega^{+}}|G_{j}^{+}|^{2}+\frac{1}{2} \int_{\Omega^{+}}\dot{v}^{+}_{1} \partial_{1}(|G_{j}^{+}|^{2}) =\sum_{j=1}^{2} \int_{\Omega^{+}} \dot{G}^{+}_{1j}\partial_{1} v  G^{+}_{j}.
	\end{equation}
	 After integrating by parts, the second terms on the left hand side of  \eqref{4.1} and \eqref{4.2} vanish, also the sum of the terms on the right hand side of  \eqref{4.2} and \eqref{4.3} vanish, thus adding \eqref{4.1}, \eqref{4.2} and \eqref{4.3} and integrating by parts, we get
	\begin{equation}\label{4.4}
		\frac{1}{2}	\partial_{t} \int_{\Omega^{+}}(c^{2}|h^{+}|^{2}+|v^{+}|^{2} +|G_{j}^{+}|^{2}) =c^{2}(\varrho_{0}) \int_{\Gamma} h^{+} v^{+} \cdot e_{2}.
	\end{equation}
	A similar result holds on $\Omega_{-}$ with the opposite sign  on the right hand side: 
	\begin{equation}\label{4.5}
		\frac{1}{2}	\partial_{t} \int_{\Omega^{-}}(c^{2} |h^{-}|^{2}+|v^{-}|^{2}+|G_{j}^{-}|^{2} ) =-c^{2}(\varrho_{0}) \int_{\Gamma} h^{-} v^{-}\cdot e_{2}.
	\end{equation}

	Adding \eqref{4.4} and \eqref{4.5} implies 
	\begin{equation}\label{4.6}
		\frac{1}{2}	\partial_{t} \int_{\Omega}(c^{2} |h|^{2}+|v|^{2}+|G_{j}|^{2} ) =c^{2}\int_{\Gamma} [h v \cdot e_{2} ]= 2c^{2}(\varrho_{0}) \int_{\Gamma} h \bar{v}^{+}_{1} \partial_{1} f.
	\end{equation}
	Also multiplying the fourth  equation in \eqref{3.5} by $f$, we have 
	\begin{equation}\label{4.7}
		\frac{1}{2}	\partial_{t}  \int_{\Gamma}|f|^{2} = \int_{\Gamma} v_{2} f.
	\end{equation}
	Adding \eqref{4.6} and \eqref{4.7} and using the Holder inequality  yields
	\begin{equation}\label{4.8}
		\begin{aligned}
			&\frac{1}{2}	\partial_{t} \int_{\Omega}(c^{2}|h|^{2}+|v|^{2}+|G_{j}|^{2} ) +	\frac{1}{2}	\partial_{t}  \int_{\Gamma}|f|^{2}\\
			&= 2c^{2}\int_{\Gamma} h \dot{v}^{+}_{1} \partial_{1} f+\int_{\Gamma} v_{2} f  \\
			&\leq 2c^{2}\dot{v}^{+}_{1}  \|h\|_{L^{2}(\Gamma)} \|\partial_{1} f\|_{L^{2}(\Gamma)}+ \| v_{2} \|_{L^{2}(\Gamma)} \|f\|_{L^{2}(\Gamma)}:=J.
		\end{aligned}
	\end{equation}
To avoid the loss of derivatives, we suppose that the solutions are band-limited at radius $R>0$, i.e., that 
$$	\underset{x_{2}\in \mathbb{R}}{\cup}supp(|\hat{f}(\cdot)|+ |\hat{h}(\cdot,x_{2})|+|\hat{v}(\cdot,x_{2})|+|\hat{G_{j}}(\cdot,x_{2})|) \subset B(0,R),$$
also we introduce an   anisotropic trace estimate in Lemma B.1 (\cite{Xin}):
\begin{equation}\label{4.9}
	\| \phi\|^{2}_{L^{2}(\Gamma)} \leq C(\|\dot{v}\cdot \nabla \phi\|_{L^{2}(\Omega)}\| \phi\|_{L^{2}(\Omega)}+ \| \phi\|^{2}_{L^{2}(\Omega)}),
\end{equation}
where $\dot{v}=(\dot{v}_1,0)$ with $\dot{v}^{+}_{1}=U_{low}+c\epsilon_{0}>0$.
Now we estimate $J$ as follows:
\begin{equation}\label{4.10}
	\begin{aligned}
		J&\lesssim (\|\dot{v}^{+}_{1}\partial_{1} h\|_{L^{2}(\Omega)}\| h\|_{L^{2}(\Omega)}+ \| h\|^{2}_{L^{2}(\Omega)}) ^{\frac{1}{2}}  \|\eta \hat{f}\|_{L^{2}(\Gamma)}\\
		&+  (\|\dot{v}^{+}_{1}\partial_{1} v_{2}\|_{L^{2}(\Omega)}\| v_{2}\|_{L^{2}(\Omega)}+ \| v_{2}\|^{2}_{L^{2}(\Omega)})^{\frac{1}{2}} \|f\|_{L^{2}(\Gamma)}\\
		&\lesssim ((\dot{v}^{+}_{1}R+1)R^2 \| h\|^{2}_{L^{2}(\Omega)}) ^{\frac{1}{2}}  \| f\|_{L^{2}(\Gamma)}+  ((\dot{v}^{+}_{1}R+1) \| v_{2}\|^{2}_{L^{2}(\Omega)})^{\frac{1}{2}} \|f\|_{L^{2}(\Gamma)}.
	\end{aligned}
\end{equation}
Finally plugging \eqref{4.10} into \eqref{4.8} and taking use of Gronwall's inequality, for arbitrary $R$, we have 
	\begin{equation}\label{4.11}
		\| f\|^{2}_{L^{2}(\Gamma)} + \| h\|^{2}_{L^{2}(\Omega)}+ \| v\|^{2}_{L^{2}(\Omega)}+ \|G_{j}\|^{2}_{L^{2}(\Omega)}\leq C(\| f_{0}\|^{2}_{L^{2}(\Gamma)} + \| h_{0}\|^{2}_{L^{2}(\Omega)}+ \| v_{0}\|^{2}_{L^{2}(\Omega)}+ \| G_{0,j}\|^{2}_{L^{2}(\Omega)}).
	\end{equation}
	From this, we infer that if $(f,h,v,G_{j})|_{t=0}=0$, then it follows that $(f,h,v,G_{j})\equiv0$. 
	
\end{proof}

\subsection{Discontinuous dependence on the initial data}
\quad  In according  with Lemma 3.2 and \eqref{3.34}, if $\sqrt{(\dot{G}^{+}_{11})^{2}+(\dot {G}^{+}_{12})^{2}}<\dot{v}^{+}_{1}<\sqrt{2c^{2}+ (\dot{G}^{+}_{11})^{2}+(\dot{G}^{+}_{12})^{2}}$, we deduce that $X_{1}^{2}$ is positive, it follows that $\tau =    X_{1} \eta$ is  real and positive. Next, we claim that the equation \eqref{3.7} can be simplied to the following form
\begin{equation}\label{4.12}
	\partial_{t}^{2}f + 	\lambda \partial^{2}_{1} f=0.
\end{equation}
where $\lambda$ must be positive in the case of  $\sqrt{(\dot{G}^{+}_{11})^{2}+(\dot {G}^{+}_{12})^{2}}<\dot{v}^{+}_{1}<\sqrt{2c^{2}+ (\dot{G}^{+}_{11})^{2}+(\dot{G}^{+}_{12})^{2}}$. In fact, plugging $f= e^{\tau t} g$ into  equation \eqref{4.12}, we get $\tau^{2} g + 	\lambda \partial^{2}_{x} g=0$. Then we take Fourier transform  of this identity with respect the $x_1$ variable, we have
\begin{equation}\label{4.13}
	(\tau^{2} -	\lambda \eta^{2}) \hat{g} =0,
\end{equation}
which yields   $\lambda = \frac{\tau^{2}}{\eta^{2}}$. From Lemma 3.2, we know that  $X_{1}^{2}=\frac{\tau^{2}}{\eta^2}>0$ in the case of  $\sqrt{(\dot{G}^{+}_{11})^{2}+(\dot {G}^{+}_{12})^{2}}<\dot{v}^{+}_{1}<\sqrt{2c^{2}+ (\dot{G}^{+}_{11})^{2}+(\dot{G}^{+}_{12})^{2}}$. Thus we have $\lambda=X_{1}^{2}>0$. Therefore  \eqref{3.7} is indeed a linear elliptic equation. 

Firstly, we state the following lemma which gives the explicit expression of the solutions $\hat{w}_1,\hat{w}_2$ and $\hat{E}_{1j}, \hat{E}_{2j}$ in terms of $\hat{g}(\eta)$. 

\begin{lemm}
    The horizontal Fourier transform of the solution $w_{i}$ snd $E_{ij}$  with $i,j=1,2$ to \eqref{3.10} satisfies the following equality
    \begin{equation}\label{4.42}
		\hat{w}_{1}(\eta, x_{2})=\left\{
		\begin{aligned}
			& -\frac{ (\mu^{+}-\mu^{-})c^{2} i \eta (\tau+ i\dot{v}^{+}_{1} \eta) }{ (\tau+i  \dot{v}^{+}_{1} \eta)^{2}+ ((\dot{G}^{+}_{11})^{2}+(\dot{G}^{+}_{12})^{2}) \eta^{2} }\hat{g}(\eta)e^{-\mu^{+}x_2}&x_2\ge 0,\\
			&-\frac{ (\mu^{+}-\mu^{-})c^{2} i \eta (\tau- i\dot{v}^{+}_{1} \eta) }{(\tau-i  \dot{v}^{+}_{1} \eta)^{2}+ ((\dot{G}^{+}_{11})^{2}+(\dot{G}^{+}_{12})^{2}) \eta^{2} }\hat{g} (\eta)e^{\mu^{-}x_2}&x_2<0,
		\end{aligned}
		\right.
	\end{equation}
	\begin{equation}\label{4.43}
		\hat{w}_{2}(\eta, x_{2})=\left\{
		\begin{aligned}
			& \frac{(\mu^{+}-\mu^{-})c^{2} \mu^{+} (\tau+ i\dot{v}^{+}_{1} \eta) }{(\tau+i  \dot{v}^{+}_{1} \eta)^{2}+ ((\dot{G}^{+}_{11})^{2}+(\dot{G}^{+}_{12})^{2}) \eta^{2}}\hat{g}(\eta)e^{-\mu^{+}x_2}&x_2\ge 0,\\
			& -\frac{(\mu^{+}-\mu^{-})c^{2} \mu^{-} (\tau- i\dot{v}^{+}_{1} \eta) }{(\tau-i  \dot{v}^{+}_{1} \eta)^{2}+ ((\dot{G}^{+}_{11})^{2}+(\dot{G}^{+}_{12})^{2}) \eta^{2}}\hat{g}(\eta)e^{\mu^{-}x_2}&x_2<0,
		\end{aligned}
		\right.
	\end{equation}
    \begin{equation}\label{4.421}
		\hat{E}_{1j}(\eta, x_{2})=\left\{
		\begin{aligned}
			& -\frac{ i\dot{G}^+_{1j}\eta (\mu^{+}-\mu^{-})c^{2} i \eta (\tau+ i\dot{v}^{+}_{1} \eta) }{(\tau+i\dot{v}^{+}_{1}\eta) [(\tau+i  \dot{v}^{+}_{1} \eta)^{2}+ ((\dot{G}^{+}_{11})^{2}+(\dot{G}^{+}_{12})^{2}) \eta^{2} ]}\hat{g}(\eta)e^{-\mu^{+}x_2}&x_2\ge 0,\\
			&-\frac{  i\dot{G}^-_{1j}\eta (\mu^{+}-\mu^{-})c^{2} i \eta (\tau- i\dot{v}^{+}_{1} \eta) }{(\tau-i\dot{v}^{+}_{1}\eta)[(\tau-i  \dot{v}^{+}_{1} \eta)^{2}+ ((\dot{G}^{+}_{11})^{2}+(\dot{G}^{+}_{12})^{2}) \eta^{2}] }\hat{g} (\eta)e^{\mu^{-}x_2}&x_2<0,
		\end{aligned}
		\right.
	\end{equation}
	and
	\begin{equation}\label{4.431}
		\hat{E}_{2j}(\eta, x_{2})=\left\{
		\begin{aligned}
			& \frac{ i\dot{G}^+_{1j}\eta (\mu^{+}-\mu^{-})c^{2} \mu^{+} (\tau+ i\dot{v}^{+}_{1} \eta) }{(\tau+i\dot{v}^{+}_{1}\eta)[(\tau+i  \dot{v}^{+}_{1} \eta)^{2}+ ((\dot{G}^{+}_{11})^{2}+(\dot{G}^{+}_{12})^{2}) \eta^{2}]}\hat{g}(\eta)e^{-\mu^{+}x_2}&x_2\ge 0,\\
			& -\frac{ i\dot{G}^-_{1j}\eta (\mu^{+}-\mu^{-})c^{2} \mu^{-} (\tau- i\dot{v}^{+}_{1} \eta) }{(\tau-i\dot{v}^{+}_{1}\eta)[(\tau-i  \dot{v}^{+}_{1} \eta)^{2}+ ((\dot{G}^{+}_{11})^{2}+(\dot{G}^{+}_{12})^{2}) \eta^{2}]}\hat{g}(\eta)e^{\mu^{-}x_2}&x_2<0.
		\end{aligned}
		\right.
	\end{equation}
\end{lemm}

\begin{proof}
    Taking the horizontal Fourier transform of  the second and third equation in \eqref{3.10}, we arrive
	\begin{equation}\label{4.36}
		(\tau+i  \dot{v}_{1} \eta)  \hat{w}_{1}+ c^{2} i \eta  \hat{m}=i\sum_{j=1}^{2} \dot{G}_{1j}\eta \hat{E}_{1j},
	\end{equation}
	and
	\begin{equation}\label{4.37}
		(\tau+i  \dot{v}_{1} \eta)  \hat{w}_{2}+ c^{2} \partial_{2} \hat{m}=i \sum_{j=1}^{2} \dot{G}_{1j}\eta \hat{E}_{2j},
	\end{equation}
	and 
		\begin{equation}\label{4.38}
		(\tau+i  \dot{v}_{1} \eta)  \hat{E}_{1j} =i\dot{G}_{1j} \eta \hat{w}_{1},
	\end{equation}
	and
		\begin{equation}\label{4.39}
		(\tau+i  \dot{v}_{1} \eta)  \hat{E}_{2j} =i\dot{G}_{1j}\eta  \hat{w}_{2},
	\end{equation}
	Then we substitute \eqref{4.38} into \eqref{4.36} to obtain 
	\begin{equation*}\label{4.40}
		(\tau+i  \dot{v}_{1} \eta)^{2}  \hat{w}_{1}+ c^{2}	(\tau+i  \dot{v}_{1} \eta) i \eta  \hat{m}=-((\dot{G}^{+}_{11})^{2}+(\dot{G}^{+}_{12})^{2}) \eta^{2} \hat{w}_{1},
	\end{equation*}
	similarly substituting \eqref{4.39} into \eqref{4.37} to obtain
		\begin{equation*}\label{4.41}
		(\tau+i  \dot{v}_{1} \eta)^{2}  \hat{w}_{2}+ c^{2}	(\tau+i  \dot{v}_{1} \eta) \partial_{2}  \hat{m}=-((\dot{G}^{+}_{11})^{2}+(\dot{G}^{+}_{12})^{2}) \eta^{2} \hat{w}_{2}.
	\end{equation*}
	Taking use of \eqref{3.15}, we obtain \eqref{4.42} and \eqref{4.43}. Futher, we can derive \eqref{4.421} and \eqref{4.431} by \eqref{4.38} and \eqref{4.39}.
	
\end{proof}
We are now in a position to prove  ill-posedness for this linear problem \eqref{3.5}-\eqref{3.6} in the following lemma:
\begin{lemm}\label{lem:4.3}
	In the case of $\sqrt{\frac{(\dot{G}^{+}_{11})^{2}+(\dot{G}^{+}_{12})^{2}}{c^{2}}}+\epsilon_{0}\le M:=\frac{|\dot{v}^{+}_{1}|}{c}\le \sqrt{2+ \frac{(\dot{G}^{+}_{11})^{2}+(\dot{G}^{+}_{12})^{2}}{c^{2}}}-\epsilon_{0}$,	the linear equations \eqref{3.5} with the corresponding jump boundary conditions \eqref{3.6} is ill-posed in the sense of Hadamard in $H^{k}(\Omega)$ for every $k$. More precisely, for any $k, j\in \mathbb{N}$ with $j\geq k$ and for any $T_{0}>0$ and $\alpha>0$, there exists a sequence $\{ (f_{n}, v_{n}, h_{n}), G_{n,l} \}_{n=1}^{\infty}$ with $l=1,2$ to \eqref{3.5}, satisfying  boundary conditions \eqref{3.6}, so that
	\begin{equation}\label{4.15}
		\| (f_{n}(0), h_{n}(0), v_{n}(0) , G_{n,l}(0))\|_{H^{j}} \lesssim \frac{1}{n},
	\end{equation}
	but
	\begin{equation}\label{4.16}
		\|(f_{n}(t),h_{n}(t), v_{n}(t), G_{n,l}(t))\|_{H^{k}}\geq \alpha,~for~all~t\geq T_{0}.
	\end{equation}
\end{lemm}

\begin{proof}
	For any $j\in \mathbb{N}$, we let $\chi_{n}(\eta)\in C_{c}^{\infty}(\mathbb{R})$ be a real-valued function so that  $supp(\chi_{n}) \subset B(0,n+1)\backslash B(0,n)$ and 
	\begin{equation} \label{4.17}
		\int_{\mathbb{R}} (1+ |\eta|^{2})^{j+1}	|\chi_{n}(\eta)|^{2} d \eta= \frac{1}{\bar{C_{j}}^{2} n^{2}},
	\end{equation}
 where $\bar{C_{j}}$ is a constant depend on $j$.  We can choose $\hat{g}_{n}=\chi_{n}(\eta)$ which solves \eqref{4.13}.
	We define 
	\begin{equation} \label{4.18}
		f_{n}(t)= e^{\tau t} g_{n}= \frac{1}{4\pi^2}\int_{\mathbb{R}}e^{X_{1} \eta t} \chi_{n}(\eta)  e^{i \eta x_{1}}d \eta,
	\end{equation}
	which solves \eqref{4.12}. Here we take use of  $\tau=X_{1} \eta$ in according with Lemma \ref{lem:symbol} . From this, we can see that the  front of the linear equation is qualitatively more unstable for large frequencies $\eta$. Since $\eta\rightarrow \infty$, the solutions \eqref{4.12} with  a higher frequency grow faster in time, which provides a mechanism for Kelvin-Helmholtz instability.   By the choice of $\chi_{n}$, we have the estimate
	\begin{equation} \label{4.19}
		\begin{aligned}
			&\|f_{n}(0)\|_{H^{j}(\Gamma)}=\|g_{n}\|_{H^{j}(\Gamma)} \\
			&= (\int_{\mathbb{R}}  (1+ |\eta|^{2})^{j}	|\chi_{n}(\eta)|^{2} d \eta)^{1/2}\lesssim \frac{1}{n}.
		\end{aligned}
	\end{equation}
	Meanwhile  for $ n+1\geq \eta\geq n$ and $t\geq T_{0}$, we get
	\begin{equation} \label{4.20}
		\begin{aligned}
			\|f_{n}(t)\|^{2}_{H^{k}(\Gamma)} &\geq e^{ 2X_{1} n T_{0}} \int_{\mathbb{R}}  (1+ |\eta|^{2})^{k}	|\chi_{n}(\eta)|^{2} d \eta\\
			&\geq  \frac{e^{2X_{1} n T_{0}}}{ (1+(n+1)^{2})^{j-k+1}}  \int_{\mathbb{R}}  (1+\eta^{2})^{j+1}    |\chi_{n}(\eta)|^{2} d\eta.
		\end{aligned}
	\end{equation}
	Let $n$ be sufficiently large so that 
	\begin{equation}\label{4.21}
		\frac{e^{2X_{1} n T_{0}}}{ (1+(n+1)^{2})^{j-k+1}} \geq \alpha ^{2}\bar{C_{j}}^{2} n^{2},
	\end{equation}
	where $\alpha$ is some positive constant independed of $n$.
	Thus we may estimate 
	\begin{equation}\label{4.22}
		\| f_{n}(t)\|_{H^{k}(\Gamma)}\geq \alpha.
	\end{equation}

	From \eqref{3.15}, we know that 
	\begin{equation} \label{4.23}
		\hat{m}_{n}(\eta,x_{2})=\left\{
		\begin{aligned}
			& \frac{4 i \dot{v}^{+}_{1} \eta \tau }{c^{2}(\mu^{+}+\mu^{-})}\hat{g}_{n}(\eta)e^{-\mu^{+}x_2}&x_2\ge 0,\\
			& \frac{4 i \dot{v}^{+}_{1} \eta \tau  }{c^{2}(\mu^{+}+\mu^{-})}\hat{g}_{n}(\eta)e^{\mu^{-}x_2}&x_2<0.
		\end{aligned}
		\right.
	\end{equation}
	Since  $\tau=X_{1}\eta>0$ and $\eta>0$, from lemma 3.1 we know that $\mu^{+}-\mu^{-}\neq 0$, then \eqref{4.23} can be rewritten as
	\begin{equation}\label{4.24}
		\hat{m}_{n}(\eta,x_{2})=\left\{
		\begin{aligned}
			& (\mu^{+}-\mu^{-})\hat{g}_{n}(\eta)e^{-\mu^{+}x_2}&x_2\ge 0,\\
			& (\mu^{+}-\mu^{-})\hat{g}_{n}(\eta)e^{\mu^{-}x_2}&x_2<0,
		\end{aligned}
		\right.
	\end{equation}
	here we note that $\mu^{\pm}$ only depend on $\eta$, since we get $\tau= X_{1}\eta$, therefore it implies that $\mu(\tau,\eta)= \mu(X_{1}\eta,\eta)$.

	Then, by \eqref{3.40} in Lemma 3.2, we deduce that 
	\begin{equation}\label{4.26}
		\big|\frac{\mu^{+}-\mu^{-}}{\mu^{\pm }}\big|^{2}= \frac{|2i\sqrt{\frac{r-a}{2}} |^{2}}{|\sqrt{\frac{r+a}{2}} \pm i\sqrt{\frac{r-a}{2}}|^{2}} =2\frac{r-a}{r}.
	\end{equation}
In order to simplify the computation, we introduce the notation $K=\sqrt{\frac{(\dot{G}^{+}_{11})^{2}+(\dot{G}^{+}_{12})^{2}}{c^{2}}} $ and $\tilde{X}_{1}= \frac{X_{1}}{c}$, thus the condition \eqref{2.28} is rewritten as $K+\epsilon_{0}\le M\le \sqrt{ K^{2}+2}-\epsilon_{0}$ and the equality \eqref{3.34} is rewritten as $\tilde{X}_{1}^{2}=\sqrt{1+ 4(K^{2}+1)M^{2}}- M^{2}-K^{2}-1$. At the same time, we rewrite \eqref{3.41} by
	\begin{equation}\label{4.27}
		\begin{aligned}
		a&= (\tilde{X}^{2}_{1}- M^{2}+K^{2}+ 1) \eta^{2}\\
		&=(\sqrt{1+ 4(K^{2}+1)M^{2}}- 2M^{2})  \eta^{2}.
			\end{aligned}
	\end{equation}
 We estimate $a$ as follows:
	\begin{equation}\label{4.28}
		-\eta^{2}<a\leq  (\sqrt{1+4(K^{2}+1)(K+\epsilon_{0})^{2}}-2(K+\epsilon_{0})^{2})\eta^{2},
	\end{equation}
	where the lower bound in \eqref{4.28} holds when $M=\sqrt{ K^{2}+2}$ in \eqref{4.27}.
	Also we compute
	\begin{equation}\label{4.29}
		\begin{aligned}
			&|\mu^{+}|= |\sqrt{\frac{r+a}{2}} + i\sqrt{\frac{r-a}{2}}| =\sqrt{ r},\\
			&|\mu^{-}|= |\sqrt{\frac{r+a}{2}} - i\sqrt{\frac{r-a}{2}}| = \sqrt{ r},
		\end{aligned}
	\end{equation}
	In according with \eqref{3.34} and \eqref{3.41}, it implies that 
	\begin{equation}\label{4.30}
			r^{2}=a^{2}+b^{2}=\eta^{4}.
	\end{equation}
	Finally, combining with \eqref{4.26}, \eqref{4.27}, \eqref{4.28} and  \eqref{4.30} implies that 
	\begin{equation}\label{4.31}
	C_{1}:=2-2(\sqrt{1+4(K^{2}+1)(K+\epsilon_{0})^{2}}-2(K+\epsilon_{0})^{2})\leq 	|\frac{\mu^{+}-\mu^{-}}{\mu^{\pm }}|^{2} <4, 
	\end{equation}
	here we remark that  $\sqrt{\frac{(\dot{G}^{+}_{11})^{2}+(\dot{G}^{+}_{12})^{2}}{c^{2}}}+\epsilon_{0} \leq M$ must be satisfied,  where $\epsilon_{0}$ is a small but fixed number. Because if March number $M$ tends
	 to $\sqrt{\frac{(\dot{G}^{+}_{11})^{2}+(\dot{G}^{+}_{12})^{2}}{c^{2}}}$, this lower bound of \eqref{4.31} would tend to zero.\\
	At the same time, we need to estimate the term $|\frac{\mu^{\pm}}{\mu^{+}+\mu^{-}}|$, by \eqref{4.27}, \eqref{4.29} and \eqref{4.30}, after the direct computation, we can obtain 
	\begin{equation*}
		\left|\frac{\mu^{\pm}}{\mu^{+}+\mu^{-}}\right|=\frac{1}{\sqrt{2}\sqrt{\sqrt{1+4(K^2+1)M^2}+1-2M^2}}
	\end{equation*}
	Then, since $K+\epsilon_{0}\le M\le \sqrt{ K^{2}+2}-\epsilon_{0}$, it is easy to verified that 
	\begin{equation}\label{upbound}
	 \frac{1}{2}<\left|\frac{\mu^{\pm}}{\mu^{+}+\mu^{-}}\right|\le C_{*},
	\end{equation}
	where $C_{*}$ is defined by 
	$$C_{*}:=\frac{1}{\sqrt{2}}\frac{1}{\sqrt{1+\sqrt{(2K^2+3)^2-4(K^2+1)\epsilon_{0}(2\sqrt{K^2+2}-\epsilon_{0})}-2(\sqrt{K^2+2}-\epsilon_{0})^2}}.$$
	We remark that when $\epsilon_0$ goes to $0$, $C_{*}$ would go to infinity, which would not obtain the uniform upper bound.\\

Therefore employing \eqref{4.31}, \eqref{4.24} and \eqref{4.29},    we estimate $\| h_{n}(0)\|_{H^{j}(\Omega)}$ as follows
	\begin{equation}\label{4.32}
		\begin{aligned}
			&\| h_{n}(t=0)\|^{2}_{H^{j}(\Omega)}= \|  m_{n}\|^{2}_{H^{j}(\Omega)}\\
			&\leq   \sum_{s=0}^{j} \int_{\mathbb{R}}  (1+\eta^{2})^{j-s}|\mu^{+}-\mu^{-}|^{2}| \hat{g}_{n}(\eta)|^{2} \int_{0}^{\infty}|\partial_{2}^{s} e^{-\mu^{+}x_2}|^{2} dx_{2} d \eta \\
			&+   \sum_{s=0}^{j}\int_{\mathbb{R}}  (1+\eta^{2})^{j-s} |\mu^{+}-\mu^{-}|^{2} | \hat{g}_{n}(\eta)|^{2} \int_{-\infty}^{0}|\partial_{2}^{s}e^{\mu^{-}x_2}|^{2} dx_{2}  d \eta\\
				&\leq   \sum_{s=0}^{j} \int_{\mathbb{R}}  (1+\eta^{2})^{j-s}|\mu^{+}-\mu^{-}|^{2}| \hat{g}_{n}(\eta)|^{2} \int_{0}^{\infty}|\mu^+|^{2s} e^{-2\sqrt{\frac{r+a}{2}}x_2} dx_{2} d \eta \\
			&+   \sum_{s=0}^{j}\int_{\mathbb{R}}  (1+\eta^{2})^{j-s} |\mu^{+}-\mu^{-}|^{2} | \hat{g}_{n}(\eta)|^{2} \int_{-\infty}^{0}|\mu^-|^{2s} e^{2\sqrt{\frac{r+a}{2}}x_2} dx_{2}  d \eta \\
			& \leq  2C_{*} \sum_{s=0}^{j}\int_{\mathbb{R}}  (1+\eta^{2})^{j-s} \left|\frac{\mu^{+}-\mu^{-}}{\mu^{+ }}\right|^{2} |\mu^{+}|^{2s+1} 	\left|\frac{\mu^{+}}{\mu^{+}+\mu^{-}}\right||\chi_{n}(\eta)|^{2} d\eta\\
			&+ 2C_{*} \sum_{s=0}^{j} \int_{\mathbb{R}}  (1+\eta^{2})^{j-s}   \left|\frac{\mu^{+}-\mu^{-}}{\mu^{-}}\right|^{2} |\mu^{-}|^{2s+1}	\left|\frac{\mu^{-}}{\mu^{+}+\mu^{-}}\right| |\chi_{n}(\eta)|^{2} d \eta\\
			& < 4C_{*}(j+1) \int_{\mathbb{R}}  (1+\eta^{2})^{j+1}  |\chi_{n}(\eta)|^{2} d\eta  \lesssim \frac{1}{n^2}.
		\end{aligned}
	\end{equation}
	While for the lower bound of $\| h_{n}(t)\|^{2}_{H^{k}(\Omega)}$, by the definition of the $H^k$ norm and \eqref{4.23}, noting that $\mu^{+}= \sqrt{\frac{r+a}{2}} + i\sqrt{\frac{r-a}{2}}$, $\mu^{-}= \sqrt{\frac{r+a}{2}} - i\sqrt{\frac{r-a}{2}}$ and $supp(\chi_{n}) \subset B(0,n+1)\backslash B(0,n)$, by \eqref{4.29} and \eqref{4.30}, we have
	\begin{equation}\label{4.25}
		\begin{aligned}
			&\| h_{n}(t)\|^{2}_{H^{k}(\Omega)}= \| e^{\tau t} m_{n}\|^{2}_{H^{k}(\Omega)}\\
			&\geq    \int_{\mathbb{R}}  (1+\eta^{2})^{k}|\mu^{+}-\mu^{-}|^{2}|e^{\tau t} \hat{g}_{n}(\eta)|^{2} \int_{0}^{\infty}|e^{-\mu^{+}x_2}|^2 dx_{2} d \eta \\
			&+  \int_{\mathbb{R}}  (1+\eta^{2})^{k} |\mu^{+}-\mu^{-}|^{2} |e^{\tau t} \hat{g}_{n}(\eta)|^{2} \int_{-\infty}^{0}|e^{\mu^{-}x_2}|^2 dx_{2}  d \eta \\
			&\geq   \int_{\mathbb{R}}  (1+\eta^{2})^{k}\frac{|\mu^{+}-\mu^{-}|^{2}}{|\mu^{+}|^2}e^{2\tau t}| \hat{g}_{n}(\eta)|^{2} 	\left|\frac{\mu^{+}}{\mu^{+}+\mu^{-}}\right||\mu^{+}| d \eta \\
			 &+\int_{\mathbb{R}}  (1+\eta^{2})^{k}\frac{|\mu^{+}-\mu^{-}|^{2}}{|\mu^{-}|^2}e^{2\tau t}| \hat{g}_{n}(\eta)|^{2} 	\left|\frac{\mu^{-}}{\mu^{+}+\mu^{-}}\right||\mu^{-}| d \eta \\
			&\geq  \int_{\mathbb{R}}  (1+\eta^{2})^{k}|\frac{\mu^{+}-\mu^{-}}{\mu^{+ }}|^{2} e^{2 X_{1} \eta t}\frac{|\mu^{+}|}{|\mu^{+}|+|\mu^{-}|} |\mu^{+}||\chi_{n}(\eta)|^{2} d \eta \\
			&+\int_{\mathbb{R}}  (1+\eta^{2})^{k}|\frac{\mu^{+}-\mu^{-}}{\mu^{-}}|^{2} e^{2 X_{1} \eta t}\frac{|\mu^{-}|}{|\mu^{+}|+|\mu^{-}|} |\mu^{-}||\chi_{n}(\eta)|^{2} d \eta \\
		& \geq  C_{1}\int_{\mathbb{R}}  (1+\eta^{2})^{k}  e^{2 X_{1} \eta t} |\chi_{n}(\eta)|^{2} d\eta,
		\end{aligned}
	\end{equation}
	where we use the triangle inequality $|\mu^{+}+\mu^{-}|\le |\mu^{+}|+|\mu^{-}|$ in the second to the last inequality above
	. 
	Meanwhile  for $\eta\geq n \geq 1$ and $t\geq T_{0}$, we may estimate \eqref{4.25} as follows
	\begin{equation}\label{4.33}
		\begin{aligned}
			\| h_{n}(t)\|^{2}_{H^{k}(\Omega)}& \geq  C_{1}\frac{e^{2 X_{1}n T_{0}}}{ 1+(n+1)^{j-k+1}}   \int_{\mathbb{R}}  (1+\eta^{2})^{j+1}    |\chi_{n}(\eta)|^{2} d\eta,
		\end{aligned}
	\end{equation}
	Let $n$ be sufficiently large so that 
	\begin{equation}\label{4.34}
		C_{1}\frac{e^{2 X_{1}n T_{0}}}{ 1+(n+1)^{j-k+1}} \geq \alpha^{2} n^{2} \bar{C}_{j}^{2},
	\end{equation}
	where $\alpha$ is some positive constant independent of $n$.
	Hence we may estimate 
	\begin{equation}\label{4.35}
		\| h_{n}(t)\|_{H^{k}(\Omega)}\geq \alpha.
	\end{equation}
	
	To estimate $\|v_n(0)\|_{H^j(\Omega)}$ and $\|v_n(t)\|_{H^k(\Omega)}$,
	 we need to estimate  $|\frac{i\eta(\tau\pm i\dot{v}^{+}_{1} \eta )}{ (\tau\pm i  \dot{v}^{+}_{1} \eta)^{2}+ ((\dot{G}^{+}_{11})^{2}+(\dot{G}^{+}_{12})^{2}) \eta^{2}} |$ and $|   \frac{\mu^{\pm }(\tau\pm i\dot{v}^{+}_{1} \eta )}{ (\tau\pm i  \dot{v}^{+}_{1} \eta)^{2}+ ((\dot{G}^{+}_{11})^{2}+(\dot{G}^{+}_{12})^{2}) \eta^{2}}|$ as follows:
	\begin{equation}\label{4.44}
		\begin{aligned}
			&\left|\frac{i\eta(\tau\pm i\dot{v}^{+}_{1} \eta )}{ (\tau\pm i  \dot{v}^{+}_{1} \eta)^{2}+ ((\dot{G}^{+}_{11})^{2}+(\dot{G}^{+}_{12})^{2}) \eta^{2}} \right|^{2}\\
			&= \frac{X^{2}_{1} + (\dot{v}^{+}_{1})^{2}}{(X^{2}_{1} + (\dot{v}^{+}_{1})^{2})^{2}+ 2((\dot{G}^{+}_{11})^{2}+(\dot{G}^{+}_{12})^{2})(X^{2}_{1} - (\dot{v}^{+}_{1})^{2}) + ((\dot{G}^{+}_{11})^{2}+(\dot{G}^{+}_{12})^{2})^{2}}\\
			&= \frac{\tilde{X}^{2}_{1} + M^{2}}{c^{2}[(\tilde{X}^{2}_{1} + M^{2})^{2}+ 2K^{2}(\tilde{X}^{2}_{1} - M^{2}) + K^{4}]} \\
			&= \frac{\sqrt{1+4(K^{2}+1) M^{2}} -K^{2}-1}{c^{2}[(\sqrt{1+4(K^{2}+1) M^{2}} -K^{2}-1)^{2}+ 2K^{2}(\sqrt{1+4(K^{2}+1) M^{2}}-2M^{2} -K^{2}-1) + K^{4}]}, 
		\end{aligned}
	\end{equation}
 Let $z=\sqrt{1+4(K^{2}+1) M^{2}}$, the last equality in \eqref{4.44} can be written a function of $z$ as following
 \begin{equation}\label{function1}
     \frac{1}{c^2}\frac{(K^2+1)(z-K^2-1)}{z^2-2(K^2+1)z+2K^2+1}.
 \end{equation}
	Since  $K+ \epsilon_{0}\le M\le \sqrt{ K^{2}+2}-\epsilon_{0}$, then $\sqrt{1+4(K^{2}+1) (K+ \epsilon_{0})^{2}}\le z<2K^2+3$, the function in \eqref{function1} is monotone decreasing when $\sqrt{1+4(K^{2}+1) (K+ \epsilon_{0})^{2}}\le z<2K^2+3$, thus we know that 
	\begin{equation}\label{4.45}
	\frac{\sqrt{K^2+2}}{2c}<\left|\frac{i\eta(\tau\pm  i\dot{v}^{+}_{1} \eta )}{ (\tau\pm i  \dot{v}^{+}_{1} \eta)^{2}+ ((\dot{G}^{+}_{11})^{2}+(\dot{G}^{+}_{12})^{2}) \eta^{2}}  \right|\leq \frac{C_{2}}{c
	},
	\end{equation}
	where $C_{2}= \sqrt{\frac{\sqrt{1+4(K^{2}+1) (K+ \epsilon_{0})^{2}} -K^{2}-1}{c^{2}[(\sqrt{1+4(K^{2}+1) (K+ \epsilon_{0})^{2}} -K^{2}-1)^{2}+ 2K^{2}(\sqrt{1+4(K^{2}+1) (K+ \epsilon_{0})^{2}}-2(K+ \epsilon_{0})^{2} -K^{2}-1) + K^{4}]}}$ and the lower bound holds when $z=2K^2+3$ in the above inequality \eqref{function1}. Similarly we also get an estimate as follows:
 \begin{equation}\label{4.455}
			\frac{\sqrt{K^2+2}}{2c}<	|\frac{i\eta(\tau\pm  i\dot{v}^{+}_{1} \eta )}{ (\tau\pm i  \dot{v}^{+}_{1} \eta)^{2}+ ((\dot{G}^{+}_{11})^{2}+(\dot{G}^{+}_{12})^{2}) \eta^{2}}  |\leq \frac{C_{2}}{c
		},.
	\end{equation}

	Therefore employing \eqref{4.45}, \eqref{3.9}, \eqref{4.42} and \eqref{upbound}, noting that $\mu^{+}= \sqrt{\frac{r+a}{2}} + i\sqrt{\frac{r-a}{2}}$, $\mu^{-}= \sqrt{\frac{r+a}{2}} - i\sqrt{\frac{r-a}{2}}$, we deduce 
	\begin{equation} \label{4.47}
		\begin{aligned}
			&\|v_{n,1}(0)\|^{2}_{H^{j}(\Omega)}= \|  w_{n,1}\|^{2}_{H^{j}(\Omega)}\\
			&\leq   \sum_{s=0}^{j} \int_{\mathbb{R}}  (1+\eta^{2})^{j-s}\left|\frac{ (\mu^{+}-\mu^{-})c^{2} i \eta (\tau+ i\dot{v}^{+}_{1} \eta) }{ (\tau+i  \dot{v}^{+}_{1} \eta)^{2}+ ((\dot{G}^{+}_{11})^{2}+(\dot{G}^{+}_{12})^{2}) \eta^{2} }\right|^{2}|\hat{g}_{n}(\eta)|^{2} \int_{0}^{\infty}|\partial_{2}^{s} e^{-\mu^{+}x_2}|^{2} dx_{2} d \eta \\
			&+  \sum_{s=0}^{j} \int_{\mathbb{R}}  (1+\eta^{2})^{j-s} \left|\frac{ (\mu^{+}-\mu^{-})c^{2} i \eta (\tau- i\dot{v}^{+}_{1} \eta) }{ (\tau-i  \dot{v}^{+}_{1} \eta)^{2}+ ((\dot{G}^{+}_{11})^{2}+(\dot{G}^{+}_{12})^{2}) \eta^{2} }\right|^{2}| \hat{g}_{n}(\eta)|^{2} \int_{-\infty}^{0}|\partial_{2}^{s}e^{\mu^{-}x_2}|^{2} dx_{2}  d \eta \\
			& \leq c^2 C^2_{2}  \sum_{s=0}^{j}\int_{\mathbb{R}}  (1+\eta^{2})^{j-s} |\frac{\mu^{+}-\mu^{-}}{\mu^{+ }}|^{2} |\mu^{+}|^{2s+1} \left|\frac{\mu^{+}}{\mu^{+}+\mu^{-}}\right|  |\chi_{n}(\eta)|^{2} d\eta\\
    			&+ c^2 C^2_{2}  \sum_{s=0}^{j}\int_{\mathbb{R}}  (1+\eta^{2})^{j-s}   |\frac{\mu^{+}-\mu^{-}}{\mu^{-}}|^{2} |\mu^{-}|^{2s+1}  \left|\frac{\mu^{-
    			}}{\mu^{+}+\mu^{-}}\right| |\chi_{n}(\eta)|^{2} d \eta\\
			& < 8c^2C_{*}C^2_{2}(j+1)\int_{\mathbb{R}}  (1+\eta^{2})^{j+1}  |\chi_{n}(\eta)|^{2} d\eta  \lesssim  \frac{1}{ n^{2}}.
		\end{aligned}
	\end{equation}
	Computing similarly as above, then we have 
	\begin{equation}\label{4.48}
		\|v_{n,2}(0)\|^{2}_{H^{j}(\Omega)} \lesssim  \frac{1}{ n^{2}}.
	\end{equation}
	Noting that $supp(\chi_{n}) \subset B(0,n+1)\backslash B(0,n)$, by \eqref{4.29}-\eqref{4.31} and \eqref{4.45}, \eqref{upbound}, we  deduce
	\begin{equation} \label{4.49}
		\begin{aligned}
			&\|v_{n,1}(t)\|_{H^{k}(\Omega)}^{2}=\|e^{X_{1} \eta t}   w_{n,1}\|_{H^{k}(\Omega)}^{2} \\
			&\geq  \int_{\mathbb{R}}  (1+\eta^{2})^{k} \left|\frac{ (\mu^{+}-\mu^{-})c^{2} i \eta (\tau+ i\dot{v}^{+}_{1} \eta) }{ (\tau+i  \dot{v}^{+}_{1} \eta)^{2}+ ((\dot{G}^{+}_{11})^{2}+(\dot{G}^{+}_{12})^{2}) \eta^{2} }\right|^{2} e^{2 X_{1} \eta t}  |\hat{g}_{n}|^{2}\int_{0}^{\infty}|e^{-\mu^{+}x_2}|^2 dx_{2} d \eta\\
			&+\int_{\mathbb{R}}  (1+\eta^{2})^{k} \left|\frac{ (\mu^{+}-\mu^{-})c^{2} i \eta (\tau- i\dot{v}^{+}_{1} \eta) }{ (\tau-i  \dot{v}^{+}_{1} \eta)^{2}+ ((\dot{G}^{+}_{11})^{2}+(\dot{G}^{+}_{12})^{2}) \eta^{2} }\right|^{2}  e^{2 X_{1} \eta t}  |\hat{g}_{n}|^{2}\int_{-\infty}^{0}|e^{\mu^{-}x_2}|^2 dx_{2} d \eta\\
			&\geq  \frac{(K^2+2)c^2}{4}\int_{\mathbb{R}}  (1+\eta^{2})^{k} |\frac{ \mu^{+}-\mu^{-}}{\mu^{+}}|^{2}  e^{2 X_{1} \eta t} |\chi_{n}(\eta)|^{2}\left|\frac{\mu^+}{\mu^{+}+\mu^{-}}\right| |\mu^{+}|d \eta \\
			&+ \frac{(K^2+2)c^2}{4}\int_{\mathbb{R}}  (1+\eta^{2})^{k} |\frac{ \mu^{+}-\mu^{-}}{\mu^{-}}|^{2}  e^{2 X_{1} \eta t} |\chi_{n}(\eta)|^{2}\left|\frac{\mu^-}{\mu^{+}+\mu^{-}}\right| |\mu^{-}|d \eta\\
			&\geq \frac{(K^2+2)c^2C_{1}}{8}\frac{e^{2 X_{1}n T_{0}}}{ 1+(n+1)^{j-k+1}}   \int_{\mathbb{R}}  (1+\eta^{2})^{j+1}    |\chi_{n}(\eta)|^{2} d\eta,
		\end{aligned}
	\end{equation}
	where $c$ is the sound speed defined in \eqref{1.2}.
	Let $n$ be sufficiently large so that 
	\begin{equation}\label{4.50}
 \frac{(K^2+2)c^2C_{1}}{8}\frac{e^{2 X_{1}n T_{0}}}{ 1+(n+1)^{j-k+1}} \geq \alpha^{2} n^2 \bar{C}_{j}^{2},
	\end{equation}
	where $\alpha$ is some positive constant independent of $n$.
	Hence  we may estimate 
	\begin{equation}\label{4.51}
		\| v_{n,1}(t)\|_{H^{k}(\Omega)}\geq \alpha.
	\end{equation}
	Similarly we have 
	\begin{equation} \label{4.52}
		\|v_{n,2}\|_{H^{k}(\Omega)}\geq  \alpha.
	\end{equation}
 To estimate terms of $\|G_{n,l}(t,x_{1},x_{2})\|^{2}_{H^{j}}$,  we 	need to  rewrite  $|\frac{i\eta}{\tau\pm i\dot{v}^{+}_{1} \eta } |$ and $|\frac{\mu^{\pm}}{\tau\pm i\dot{v}^{+}_{1} \eta}|$ as follows: 
	\begin{equation}\label{Gesima}
		\begin{aligned}
			\left|\frac{i\eta}{\tau\pm i\dot{v}^{+}_{1} \eta )} \right|^{2}= \frac{1}{X^{2}_{1} + (\dot{v}^{+}_{1})^{2} } = \frac{1}{c^{2}  (\sqrt{1+ 4(K^{2}+1)M^{2}}-K^{2}-1)}. 
		\end{aligned}
	\end{equation}
	Since  $K+ \epsilon_{0}\le M\le \sqrt{ K^{2}+2}-\epsilon_{0}$, we know by \eqref{Gesima} that 
	\begin{equation}\label{4.3444}
		\begin{aligned}
			\frac{1}{c^{2}(K^{2}+2)} < 	\left|\frac{i\eta}{\tau\pm  i\dot{v}^{+}_{1} \eta )} \right|^{2} \leq  C_{3},
		\end{aligned}
	\end{equation}
	where $C_{3}= \frac{1}{c^{2}  (\sqrt{1+ 4(K^{2}+1)(K+ \epsilon_{0})^{2}}-K^{2}-1)} $.
	
	The estimates \eqref{4.29} and \eqref{4.3444} then imply that 
	\begin{equation}\label{5.334}
		\frac{1}{c^{2}(K^{2}+2)} < 	\left|\frac{\mu^{+}}{\tau\pm  i\dot{v}^{+}_{1} \eta} \right|^{2} = \frac{1}{X_{1}^{2}+ (\dot{v}^{+}_{1})^{2} }\leq C_{3}.
	\end{equation}
  Then, by \eqref{4.421}, making use of \eqref{4.3444} and \eqref{4.45}, we get with $l=1,2$
	\begin{equation} \label{4.366}
		\begin{aligned}
			&\quad\|G_{n,1l}(t=0,x_{1},x_{2})\|^{2}_{H^{j}(\Omega)}= \|  E_{n,1l}(x_{1},x_{2})\|^{2}_{H^{j}(\Omega)}\\
			&\leq   \sum_{s=0}^{j} \int_{\mathbb{R}}  (1+\eta^{2})^{j-s}\left|\frac{ i\dot{G}^+_{1l}\eta (\mu^{+}-\mu^{-})c^{2} i \eta (\tau+ i\dot{v}^{+}_{1} \eta) }{(\tau+i\dot{v}^{+}_{1}\eta) [(\tau+i  \dot{v}^{+}_{1} \eta)^{2}+ ((\dot{G}^{+}_{11})^{2}+(\dot{G}^{+}_{12})^{2}) \eta^{2} ]}\right|^{2}|\hat{g}_{n}(\eta)|^{2} \int_{0}^{\infty}|\partial_{2}^{s} e^{-\mu^{+}x_2}|^{2} dx_{2} d \eta \\
			&+  \sum_{s=0}^{j} \int_{\mathbb{R}}  (1+\eta^{2})^{j-s} \left|\frac{  i\dot{G}^-_{1l}\eta (\mu^{+}-\mu^{-})c^{2} i \eta (\tau- i\dot{v}^{+}_{1} \eta) }{(\tau-i\dot{v}^{+}_{1}\eta)[(\tau-i  \dot{v}^{+}_{1} \eta)^{2}+ ((\dot{G}^{+}_{11})^{2}+(\dot{G}^{+}_{12})^{2}) \eta^{2}] }\right|^2 |\hat{g}_{n}(\eta)|^{2}\int_{-\infty}^{0}|\partial_{2}^{s}e^{\mu^{-}x_2}|^{2} dx_{2}  d \eta\\
			&\leq  c^2 C_{3}C_2^2|\dot{G}^+_{1l}|^2\sum_{s=0}^{j} \int_{\mathbb{R}}  (1+\eta^{2})^{j-s}\left|\frac{ \mu^{+}-\mu^{-}}{\mu^+}\right|^{2}|\hat{g}_{n}(\eta)|^{2}|\mu^+|^{2s+2} \int_{0}^{\infty} e^{-2\sqrt{\frac{r+a}{2}}x_2} dx_{2} d \eta \\
			&+c^2 C_{3}C_2^2|\dot{G}^-_{1l}|^2\sum_{s=0}^{j} \int_{\mathbb{R}}  (1+\eta^{2})^{j-s}\left|\frac{ \mu^{+}-\mu^{-}}{\mu^-}\right|^{2}|\hat{g}_{n}(\eta)|^{2}|\mu^-|^{2s+2} \int_{-\infty}^{0} e^{2\sqrt{\frac{r+a}{2}}x_2} dx_{2} d \eta\\
			&\lesssim c^2C_{*}C_3C_2^2(j+1) \int_{\mathbb{R}}  (1+\eta^{2})^{j+1}  |\chi_{n}(\eta)|^{2} d\eta  \lesssim  \frac{1}{ n^{2}}.
			\end{aligned}
		\end{equation}

	Similarly,  by \eqref{4.431}, we have 
\begin{equation}\label{4.48}
\|G_{n,2l}(t=0,x_{1},x_{2})\|^{2}_{H^{j}(\Omega)} \lesssim  \frac{1}{ n^{2}}.
\end{equation}
Whereas for $n+1\geq \eta\geq n$ and $t\geq T_{0}$, similar process with \eqref{4.49}, we deduce
	\begin{equation} \label{4.377}
		\begin{aligned}
			&\|G_{n,1l}(t)\|_{H^{k}(\Omega)}^{2}= \| e^{X_{1}\eta t} E_{n,1l}(x_{1},x_{2})\|^{2}_{H^{k}(\Omega)}\\  &>  	\frac{(|\dot{G}^+_{1l}|^2+|\dot{G}^-_{1l}|^2) C_{1}}{8}\frac{e^{2 X_{1}n T_{0}}}{ 1+(n+1)^{j-k+1}}   \int_{\mathbb{R}}  (1+\eta^{2})^{j+1}    |\chi_{n}(\eta)|^{2} d\eta.  
		\end{aligned}
	\end{equation}
As previous, we let $n$ be sufficiently large so that 
		\begin{equation} \label{4.377}
\|G_{n,1l}(t)\|_{H^{k}(\Omega)} \geq \alpha.
	\end{equation}
Also we have 
	\begin{equation} \label{4.388}
\|G_{n,2l}(t)\|_{H^{k}(\Omega)} \geq \alpha.
	\end{equation}
		
	Collecting the estimates \eqref{4.19}, \eqref{4.32} and \eqref{4.47} gives 
	\begin{equation}\label{4.53}
		\|f_{n}(0)\|_{H^{j}(\Gamma)}+\|h_{n}(0)\|_{H^{j}(\Gamma)}+ \|v_{n}(0)\|_{H^{j}(\Omega)}+ \|G_{n,l}(0)\|_{H^{j}(\Omega)}\lesssim \frac{1}{n},
	\end{equation}
	but the estimates \eqref{4.22}, \eqref{4.35} \eqref{4.51} and \eqref{4.52} yield
	\begin{equation}\label{4.54}
		\|f_{n}\|_{H^{k}(\Gamma)}+ \|h_{n}\|_{H^{k}(\Omega)}+ \| v_{n}\|_{H^{k}(\Omega)}+\| G_{n,l}\|_{H^{k}(\Omega)} \geq \alpha,~for~all~t\geq T_{0},
	\end{equation}
	which completes the proof of Lemma \ref{4.3}.
\end{proof}

\section{Ill-posedness  for the nonlinear problem}
\quad \quad  Now  we will prove nonlinear ill-posedness for the nonlinear problem \eqref{2.8}.  To begin with, we rewrite the nonlinear system \eqref{2.8} in a perturbation form around the steady state.  Let
\begin{equation}\label{5.1}
	\begin{aligned}
		f=0+ \tilde{f},~v= \dot{v}+ \tilde{v},~h=\dot{h}+ \tilde{h}, \\
		\Psi= Id+ \tilde{\Psi},\varrho=\dot{\varrho}+ \tilde{\varrho},~\psi= 0+ \tilde{\psi},\\
		G_{j}=\dot{G}_{j} + \tilde{G}_{j},~ n=e_{2}+ \tilde{n},~     A=I-B,
	\end{aligned}
\end{equation}
where $	G_{j}$ means the $jth$ column of the $2\times2$ matrix $(G_{ij})$ with $i,j=1,2$, i.e. $G_j=(G_{1j}, G_{2j})^T$.  $\psi$ and $\Psi$ are defined by \eqref{2.2} and \eqref{2.5}, respectively. Meanwhile, $B$ can be represented as following
\begin{equation}\label{5.2}
	B= \sum_{n=1}^{\infty} (-1)^{n-1} (D\tilde{\Psi} )^{n}.
\end{equation}
We can rewrite the term $\breve{v}$ as follows
\begin{equation}\label{5.3}
	\begin{aligned}
		\breve{v}&= (I-B)(\dot{v}+ \tilde{v})- (0, \frac{\partial_{t}\tilde{\psi}}{1+ \partial_{2} \tilde{\psi}})\\
		&= \dot{v}+ \tilde{v}- B(\dot{v}+ \tilde{v})- (0, \frac{\partial_{t}\tilde{\psi}}{1+ \partial_{2} \tilde{\psi}}):= \dot{v}+M,
	\end{aligned}
\end{equation}
where the $M$ is defined as follows
\begin{equation}\label{5.4}
	\begin{aligned}
		M= \tilde{v}- B(\dot{v}+ \tilde{v})- (0, \frac{\partial_{t}\tilde{\psi}}{1+ \partial_{2} \tilde{\psi}}).
	\end{aligned}
\end{equation}

Similarly we rewrite the term $\breve{G}_{j}$ as follows
\begin{equation}\label{5.3}
	\begin{aligned}
		\breve{G}_{j}&= (I-B)(\dot{G}_{j}+ \tilde{G}_{j})\\
		&=\dot{G}_{j}+ \tilde{G}_{j}- B(\dot{G}_{j}+ \tilde{G}_{j}):= \dot{G}_{j}+N,
	\end{aligned}
\end{equation}
where the $N$ is defined as follows
\begin{equation}\label{5.4}
	\begin{aligned}
		N= \tilde{G}_{j}- B(\dot{G}_{j}+ \tilde{G}_{j}).
	\end{aligned}
\end{equation}

To linearize  the term $ c^{2}(h)= c^{2}(\dot{h}+ \tilde{h})$,  we employ Taylor formula to get
\begin{equation}\label{5.5}
	c^{2}(\dot{h}+ \tilde{h})= c^{2}(\dot{h}) + \mathcal{R},
\end{equation}
where the remainder term is defined by
\begin{equation}\label{5.6}
	\begin{aligned}
		\mathcal{R}=  (c^{2})^{\prime}(\dot{h}+ (1-\alpha)\tilde{h}) \tilde{h}, ~0<\alpha<1.
	\end{aligned}
\end{equation}
For the term $v\cdot n $ and $G_{j}\cdot n$,  we can rewrite it as
\begin{equation}\label{5.7}
	v\cdot n= (\dot{v}+ \tilde{v})\cdot (e_{2}+ \tilde{n})
	=\tilde{v}_{2}-\dot{v}_{1} \partial_{1} \tilde{f}+\tilde{v} \cdot \tilde{n},
\end{equation}
and
\begin{equation}\label{5.7}
	G_{j}\cdot n= (\dot{G}_{j}+ \tilde{G}_{j})\cdot (e_{2}+ \tilde{n})
	=\tilde{G}_{2j}-\dot{G}_{1j} \partial_{1} \tilde{f}+\tilde{G}_{j} \cdot \tilde{n}.
\end{equation}

Then the nonlinear system \eqref{2.8}  can be rewritten a linear system  of the perturbation terms $(\tilde{h}, \tilde{v}, \tilde{f}, \tilde{G}_{j})$ as following 
\begin{equation}\label{5.8}
	\begin{cases}
		\partial_{t}\tilde{h}+(\dot{v} \cdot \nabla) \tilde{h}+ \nabla \cdot \tilde{v}=-(M \cdot \nabla) \tilde{h}+B^{T} \nabla \cdot \tilde{v} & \text {in}~ \Omega, \\
		\partial_{t} \tilde{v}+(\dot{v} \cdot \nabla) \tilde{v}
		+ c^{2}(\dot{h})\nabla \tilde{h}-\sum_{j=1}^{2} (  \dot{G}_{j}\cdot \nabla )\tilde{G}_{j}=-(M \cdot \nabla) \tilde{v}\\
		+\sum_{j=1}^{2} (  N\cdot \nabla )\tilde{G}_{j}+c^{2}(\dot{h})B^{T}\nabla \tilde{h}- \mathcal{R}(\nabla \tilde{h}- B^{T}\nabla \tilde{h}) & \text {in}~ \Omega, \\
		\partial_{t} \tilde{G}_{j} +(\dot{v} \cdot \nabla) \tilde{G}_{j}-(\dot{G}_{j}\cdot \nabla )\tilde{v}=-(M \cdot \nabla) \tilde{G}_{j}+ (N\cdot \nabla )\tilde{v}   & \text { on }~ \Gamma,\\
		\partial_{t} \tilde{f}+\dot{v}_{1} \partial_{1} \tilde{f}-\tilde{v}_{2} =\tilde{v} \cdot \tilde{n}   & \text { on } \Gamma.
	\end{cases}
\end{equation}
The jump conditions take new form in terms of $\tilde{h}, \tilde{v}, \tilde{G}_{j}$
\begin{equation}\label{5.9}
	\begin{cases}
		\left(\tilde{v}^{+}-\tilde{v}^{-}\right) \cdot e_{2}+(\dot{v}^{+}-\dot{v}^{-})\cdot \tilde{n} =-(\tilde{v}^{+}-\tilde{v}^{-})\cdot \tilde{n}  & \text { on } \Gamma, \\
		\tilde{G}_{j}^{+} \cdot e_{2}+\dot{G}_{j}^{+}\cdot \tilde{n} +\tilde{G}_{j}^{+}\cdot \tilde{n}=0, ~\tilde{G}_{j}^{-} \cdot e_{2}+\dot{G}_{j}^{-}\cdot \tilde{n} +\tilde{G}_{j}^{-}\cdot \tilde{n}=0 & \text { on } \Gamma,\\
		\dot{h}^{+}+  \tilde{h}^{+}=\dot{h}^{-}+  \tilde{h}^{-} & \text { on } \Gamma.
	\end{cases}
\end{equation}

\underline{\textbf{Proof of Theorem \ref{theorem:main}}}\quad 
Now we are ready to prove the main Theorem 2.3.   We prove it by the method of  contradiction. Suppose that the system \eqref{2.8}  is locally well-posedness  for some $k \geq 3$. Let $\delta, t_{0}, C>0$ be the constants provided by Definition 2.2. For $\varepsilon>0$, let $(f^{\varepsilon},h^{\varepsilon},v^{\varepsilon},G^{\varepsilon}_{j})(t)$ with initial data $(f^{\varepsilon},h^{\varepsilon},v^{\varepsilon},G^{\varepsilon}_{j})|_{t=0}=(f^{\varepsilon}_{0},h^{\varepsilon}_{0},v^{\varepsilon}_{0},G^{\varepsilon}_{0,j})$ is a solution sequence of the system \eqref{2.8}. We choose $(f^{1},h^{1},v^{1},G^{1}_{j}) $ to be  $(f^{\varepsilon},h^{\varepsilon},v^{\varepsilon},G^{\varepsilon}_{j})$.  We also replace $(f^{2}_{0},h^{2}_{0},v^{2}_{0},G^{2}_{0,j}) $ by the rectilinear solution $U\equiv (\dot{f},\dot{h},\dot{v},\dot{G}_{j})$ defined in section 1.1. Obviously, $U$ is always the solution of the system \eqref{2.8}. For simplicity, we always take this steady-state $U$ as the solution of the system \eqref{2.8}, i.e.,  $(f^{2},h^{2},v^{2},G^{2}_{j})(t) =U$ for $t\geq 0$.

Fix $n \in \mathbb{N}$ so that $n>C$. Applying Lemma 4.2 with this $n, T_{0}=t_{0} / 2, k \geq 3$, and $\alpha=2$, we can find $f^{L}, h^{L}, v^{L},G^{L}_{j}$ solving \eqref{3.5} so that
\begin{equation}\label{5.12}
	\|(f^{L}_{0},h^{L}_{0}, v^{L}_{0},G^{L}_{0,j})\|_{H^{k}}\lesssim \frac{1}{n},
\end{equation}
but
\begin{equation}\label{5.13}
	\|(f^{L}(t),h^{L}(t), v^{L}(t),G^{L}_{j}(t))\|_{H^{3}} \geq 2 \quad \text { for } t \geq t_{0} / 2.
\end{equation}

We define $\tilde{f}_{0}^{\varepsilon}=f_{0}^{\varepsilon}-\dot{f} :=\varepsilon f^{L}_{0}$, $\tilde{h}_{0}^{\varepsilon}=h_{0}^{\varepsilon}-\dot{h}:=\varepsilon h^{L}_{0}$, $\tilde{v}_{0}^{\varepsilon}=v_{0}^{\varepsilon}-\dot{v}:=\varepsilon v^{L}_{0}$  and $\tilde{G}_{0,j}^{\varepsilon}=G_{0,j}^{\varepsilon}-\dot{G}_{j}:=\varepsilon G^{L}_{0,j}$. Then for $\varepsilon<\delta n$, we have $\|(\tilde{f}_{0}^{\varepsilon}, \tilde{h}_{0}^{\varepsilon},  \tilde{v}_{0}^{\varepsilon},\tilde{G}_{0,j}^{\varepsilon})\|_{H^{k}}<\delta$, so according to Definition 2.2, there exists a solution $\left(\tilde{f}^{\varepsilon}:=f^{\varepsilon}- \dot{f} , \tilde{h}^{\varepsilon}:=h^{\varepsilon}-\dot{h}, \tilde{v}^{\varepsilon}:= v^{\varepsilon}-\dot{v},\tilde{G}^{\varepsilon}_{j}:= G^{\varepsilon}_{j}-\dot{G}_{j} \right) \in L^{\infty}\left(\left[0, t_{0}\right] ; H^{3}(\Omega)\right)$ that solves \eqref{5.8}-\eqref{5.9} with  initial data  $(\tilde{f}_{0}^{\varepsilon}, \tilde{h}_{0}^{\varepsilon}, \tilde{v}_{0}^{\varepsilon},\tilde{G}_{0,j}^{\varepsilon})$  and satisfies the inequality
\begin{equation}\label{5.14}
	\begin{aligned}
		\sup _{0 \leq t \leq t_{0}}\left\|\left(\tilde{f}^{\varepsilon}, \tilde{h}^{\varepsilon}, \tilde{v}^{\varepsilon}, \tilde{G}^{\varepsilon}_{j} \right)(t)\right\|_{H^{3}} & \leq C\left(\left\|\left(f_{0}^{\varepsilon}, h_{0}^{\varepsilon}, v_{0}^{\varepsilon},\tilde{G}_{0,j}^{\varepsilon}\right)\right\|_{H^{k}}\right)  \\
		& \leq C \varepsilon \frac{1}{n}<\varepsilon.
	\end{aligned}
\end{equation}

Now define the rescaled functions $\bar{f}^{\varepsilon}=\tilde{f}^{\varepsilon} / \varepsilon, \bar{h}^{\varepsilon}=\tilde{h}^{\varepsilon} / \varepsilon , \bar{v}^{\varepsilon}=\tilde{v}^{\varepsilon} / \varepsilon,\bar{G}^{\varepsilon}_{j}=\tilde{G}^{\varepsilon}_{j} / \varepsilon$; rescaling \eqref{5.14} then shows that
\begin{equation}\label{5.15}
	\sup _{0 \leq t \leq t_{0}}\left\|(\bar{f}^{\varepsilon}, \bar{h}^{\varepsilon}, \bar{v}^{\varepsilon},\bar{G}^{\varepsilon}_{j})(t)\right\|_{H^{3}}<1.
\end{equation}

By construction, we know that  $(\bar{f}^{\varepsilon}_{0}, \bar{h}^{\varepsilon}_{0}, \bar{v}^{\varepsilon}_{0},\bar{G}^{\varepsilon}_{0j})=(f^{L}_{0}, h^{L}_{0}, v^{L}_{0},G^{L}_{0j})$. Next, we are going  to show that the rescaled functions  $(\bar{f}^{\varepsilon}, \bar{h}^{\varepsilon}, \bar{v}^{\varepsilon},\bar{G}^{\varepsilon}_{j})$ converge as $\varepsilon \rightarrow 0$ to the solutions $(f^{L}, h^{L}, v^{L},G^{L}_{j})$ of the linearized equations \eqref{3.2}.

Now we are going to  reformulate \eqref{5.8}-\eqref{5.9} in terms of rescaled functions $(\bar{f}^{\varepsilon}, \bar{h}^{\varepsilon}, \bar{v}^{\varepsilon},\bar{G}^{\varepsilon}_{j})$ and   show  some convergence results. The third equation in \eqref{5.5} can be rewritten in terms of rescaled function $(\bar{f}^{\varepsilon}, \bar{h}^{\varepsilon}, \bar{v}^{\varepsilon},\bar{G}^{\varepsilon}_{j})$  as follows:
\begin{equation}\label{5.16}
	\partial_{t} \bar{f} ^{\varepsilon}+\dot{v}_{1} \partial_{1} \bar{f}^{\varepsilon}-\bar{v}^{\varepsilon}_{2} =\varepsilon \bar{v}^{\varepsilon} \cdot n^{\varepsilon}. 
\end{equation}
where  $n^{\varepsilon}=\frac{(-\varepsilon \partial_{1} \bar{f}^{\varepsilon}, 0)}{\varepsilon}=(- \partial_{1} \bar{f}^{\varepsilon}, 0)$ is well defined and uniformly bounded in $L^{\infty}\left(\left[0, t_{0}\right] ; H^{2}(\Gamma)\right)$ since
\begin{equation}\label{5.17}
	\|n^{\varepsilon}\|_{H^{2}(\Gamma)}\leq \| \bar{f}^{\varepsilon}\|_{H^{3}(\Gamma)}<1,
\end{equation}
where the above inequality holds according to \eqref{5.15}.
Hence, by \eqref{5.15} and\eqref{5.17}, we obtain
\begin{equation}\label{5.18}
	\lim _{\varepsilon \rightarrow 0} \sup _{0 \leq t \leq t_{0}}\left\|\partial_{t} \bar{f}^{\varepsilon}+\dot{v}_{1} \partial_{1} \bar{f}^{\varepsilon}-\bar{v}^{\varepsilon}_{2}\right\|_{H^{2}}=0 
\end{equation}
and 
\begin{equation}\label{5.188}
	\sup _{0 \leq t \leq t_{0}}\left\|\partial_{t} \bar{f}^{\varepsilon}(t)\right\|_{H^{2}}\leq |\dot{v}_{1}| \sup _{0 \leq t \leq t_{0}}\left\|\partial_{1} \bar{f}^{\varepsilon}(t)\right\|_{H^{2}}+\sup _{0 \leq t \leq t_{0}}\left\|\bar{v}^{\varepsilon}_{2}\right\|_{H^{2}}\leq C
\end{equation}

Expanding the first equation in \eqref{5.8}  implies that
\begin{equation}\label{5.19}
	\partial_{t}\bar{h}^{\varepsilon}+(\dot{v} \cdot \nabla) \bar{h}^{\varepsilon}+ \nabla \cdot \bar{v}^{\varepsilon}=-\varepsilon(M^{\varepsilon} \cdot \nabla) \bar{h}^{\varepsilon}+\varepsilon(B^{\varepsilon})^{T} \nabla \cdot \bar{v}^{\varepsilon}, 
\end{equation}
where we define $M^{\varepsilon}$ as follows
\begin{equation}\label{5.20}
	M^{\varepsilon}=\bar{v}^{\varepsilon}-B^{\varepsilon}(\dot{v}+ \varepsilon \bar{v}^{\varepsilon})-(0,\frac{\partial_{t} \psi^{\varepsilon}}{1+ \varepsilon \partial_{2} \psi^{\varepsilon}}), ~\text{with}~\psi^{\varepsilon}= \theta \bar{f}^{\varepsilon} ~\text{defined in \eqref{2.2} }.
\end{equation}
In order to estimate the bound of $M^{\varepsilon}$, we firstly estimate the bound of $B^{\varepsilon}$. We  assume that $\varepsilon$ is sufficiently small so that
$\varepsilon<1 /\left(2 C_{1}\right)$, where $C_{1}>0$ is the best constant in the inequality $\|U V\|_{H^{2}} \leq$ $C_{1}\|U\|_{H^{2}}\|V\|_{H^{2}}$ for $3 \times 3$ matrix-valued functions $U, V$. This assumption guarantees that $B^{\varepsilon}:=(I-(I+\varepsilon \nabla  \Psi^{\varepsilon})^{-1})/\varepsilon $ is well defined and uniformly bounded in $L^{\infty}\left(\left[0, t_{0}\right] ; H^{2}(\Omega)\right)$ since
\begin{equation}\label{5.21}
	\begin{aligned}
		\left\|{B}^{\varepsilon}\right\|_{H^{2}} & =\left\|\sum_{n=1}^{\infty}(-\varepsilon)^{n-1}(\nabla  \Psi^{\varepsilon})^{n}\right\|_{H^{2}} \leq \sum_{n=1}^{\infty} \varepsilon^{n-1}\left\|(\nabla  \Psi^{\varepsilon})^{n}\right\|_{H^{2}}  \\
		& \leq \sum_{n=1}^{\infty}\left(\varepsilon C_{1}\right)^{n-1}\left\|\nabla  \Psi^{\varepsilon}\right\|_{H^{2}}^{n} \leq \sum_{n=1}^{\infty} \frac{1}{2^{n-1}}\left\|\psi^{\varepsilon}\right\|_{H^{3}}^{n}\\
		&\leq \sum_{n=1}^{\infty} \frac{1}{2^{n-1}}\left\|\bar{f}^{\varepsilon}\right\|_{H^{3}}^{n}
		<\sum_{n=1}^{\infty} \frac{1}{2^{n-1}}=2,
	\end{aligned}
\end{equation}
whereas we shows that 
\begin{equation}\label{5.22}
	\begin{aligned}
		\left\|{M}^{\varepsilon}\right\|_{H^{2}}& \leq \| \bar{v}^{\varepsilon}\|_{H^{2}}+  |\dot{v}| \|  B^{\varepsilon}\|_{H^{2}} +\varepsilon \|  B^{\varepsilon}\|_{H^{2}}\|\bar{v}^{\varepsilon}\|_{H^{2}}+\|\partial_{t}  \psi^{\varepsilon}\|_{H^{2}}\\
		&\leq \| \bar{v}^{\varepsilon}\|_{H^{2}}+   |\dot{v}|  \|  B^{\varepsilon}\|_{H^{2}} +\varepsilon \|  B^{\varepsilon}\|_{H^{2}}\|\bar{v}^{\varepsilon}\|_{H^{2}}+\|\partial_{t}  \bar{f}^{\varepsilon}\|_{H^{2}}\\
		&\leq C.
	\end{aligned}
\end{equation}

Therefore, by employing \eqref{5.15},\eqref{5.21} and \eqref{5.22}, we get by \eqref{5.19}
\begin{equation}\label{5.23}
	\lim _{\varepsilon \rightarrow 0}	\sup _{0 \leq t \leq t_{0}}\left\|\partial_{t} \bar{h}^{\varepsilon}+(\dot{v} \cdot \nabla) \bar{h}^{\varepsilon}+ \nabla \cdot \bar{v}^{\varepsilon} \right\|_{H^{2}}=0, 
\end{equation}
and 
\begin{equation}\label{5.24}
	\sup _{0 \leq t \leq t_{0}}\left\|\partial_{t} \bar{h}^{\varepsilon}(t)\right\|_{H^{2}}<C.
\end{equation}

Expanding the second equation in \eqref{5.8}, we find that
\begin{equation}\label{5.251}
	\begin{aligned}
		&\partial_{t} \bar{v}^{\varepsilon}+(\dot{v} \cdot \nabla) \bar{v}^{\varepsilon}+ c^{2}\nabla \bar{h}^{\varepsilon}-\sum_{j=1}^{2} (  \dot{G}_{j} \cdot \nabla )\tilde{G}_{j}=-\varepsilon(M^{\varepsilon} \cdot \nabla) \bar{v}^{\varepsilon}\\
		&+\sum_{j=1}^{2}\varepsilon (  N^{\varepsilon}\cdot \nabla )\tilde{G}_{j}+\varepsilon c^{2}(B^{\varepsilon})^{T}\nabla \bar{h}^{\varepsilon}+ \varepsilon\mathcal{R}^{\varepsilon}(\nabla  \bar{h}^{\varepsilon}- \varepsilon(B^{\varepsilon})^{T}\nabla  \bar{h}^{\varepsilon}), 
	\end{aligned}
\end{equation}
where  we define $N^{\varepsilon}$ as follows
\begin{equation}
	N^{\varepsilon}=\bar{G}^{\varepsilon}_{j}-B^{\varepsilon}(\dot{G}_{j}+ \varepsilon \bar{G}^{\varepsilon}_{j}).
\end{equation}
Making full use of \eqref{5.23},  we show that 
\begin{equation}
	\begin{aligned}
		\left\|{N}^{\varepsilon}\right\|_{H^{2}}& \leq \| \bar{G}^{\varepsilon}_{j}\|_{H^{2}}+  |\dot{G}_{j}| \|  B^{\varepsilon}\|_{H^{2}} +\varepsilon \|  B^{\varepsilon}\|_{H^{2}}\|\bar{G}^{\varepsilon}_{j}\|_{H^{2}}\\
		&\leq C.
	\end{aligned}
\end{equation}

We also define the normalized remainder function by
\begin{equation}\label{5.26}
	\mathcal{R}^{\varepsilon}(x, t)  =\frac{ (c^{2})^{\prime}(\dot{h}+ (1-\alpha) \varepsilon \bar{h}^{\varepsilon}) \varepsilon \bar{h}^{\varepsilon}}{\varepsilon}= (c^{2})^{\prime}(\dot{h}+ (1-\alpha) \varepsilon \bar{h}^{\varepsilon})\bar{h}^{\varepsilon}.
\end{equation}
It is easy to show  that $\dot{h}+ (1-\alpha) \varepsilon \bar{h}^{\varepsilon}$ is bounded above  by a positive constant. Taking use of  \eqref{5.15} implies
\begin{equation}\label{5.28}
	\sup _{0 \leq t \leq t_{0}}\left\|	\mathcal{R}^{\varepsilon}(x, t)\right\|_{H^{3}} \leq C.
\end{equation}

Therefore, from \eqref{5.28} and \eqref{5.15}, we deduce  by \eqref{5.251} that
\begin{equation}\label{5.29}
	\lim _{\varepsilon \rightarrow 0} \sup _{0 \leq t \leq t_{0}}\left\|\partial_{t} \bar{v}^{\varepsilon}+(\dot{v} \cdot \nabla) \bar{v}^{\varepsilon}+ c^{2}\nabla \bar{h}^{\varepsilon}-\sum_{j=1}^{2} (  \dot{G}_{j} \cdot \nabla )\tilde{G}_{j}\right\|_{H^{2}}=0 
\end{equation}
and
\begin{equation}\label{5.30}
	\sup _{0 \leq t \leq t_{0}}\left\|\partial_{t} \bar{v}^{\varepsilon}(t)\right\|_{H^{2}}<C.
\end{equation}

Finally, we expand  the third  equation in \eqref{5.8} to find that
\begin{equation}\label{5.25}
	\begin{aligned}
		\partial_{t} \bar{G}^{\varepsilon}_{j}+(\dot{v} \cdot \nabla) \bar{G}^{\varepsilon}_{j}- (  \dot{G}_{j} \cdot \nabla )\bar{v}^{\varepsilon}=-\varepsilon(M^{\varepsilon} \cdot \nabla) \bar{G}^{\varepsilon}_{j}+\varepsilon (  N^{\varepsilon}\cdot \nabla )\bar{v}^{\varepsilon}. 
	\end{aligned}
\end{equation}
In according with  \eqref{5.22} and \eqref{5.29}, we deduce that 
\begin{equation}\label{bound1}
	\sup _{0 \leq t \leq t_{0}}\left\|(M^{\varepsilon} \cdot \nabla) \bar{G}^{\varepsilon}_{j}+(  N^{\varepsilon}\cdot \nabla )\bar{v}^{\varepsilon}\right\|_{H^{2}} \leq C.
\end{equation}
Therefore from \eqref{5.28}, \eqref{5.15} and \eqref{bound1}, we deduce that 
\begin{equation}
	\lim _{\varepsilon \rightarrow 0} \sup _{0 \leq t \leq t_{0}}\left\|\partial_{t} \bar{G}^{\varepsilon}_{j}+(\dot{v} \cdot \nabla) \bar{G}^{\varepsilon}_{j}- (  \dot{G}_{j} \cdot \nabla )\bar{v}^{\varepsilon}\right\|_{H^{2}}=0 
\end{equation}
and
\begin{equation}
	\sup _{0 \leq t \leq t_{0}}\left\|\partial_{t} \bar{G}^{\varepsilon}_{j}(t)\right\|_{H^{2}}<C.
\end{equation}

Next, we deal with some convergence results for the jump conditions. For the first equation in \eqref{5.9} we rewrite the normal vector $n$  as follows
\begin{align*}\label{5.36}
	n=e_{2}+ \tilde{n}^{\varepsilon} : =e_{2}+\varepsilon n^{\varepsilon},~n^{\varepsilon}=(-\partial_{1} \bar{f}^{\varepsilon},0).
\end{align*}
Noting that $\dot{v}^{+}\cdot e_{2}=0$ and $\dot{G}_{j}^{+}\cdot e_{2}=0$, so we may rewrite the second equation in \eqref{5.9} as
\begin{equation}\label{5.37}
	\left(\bar{v}^{+,\varepsilon}-\bar{v}^{-,\varepsilon}\right) \cdot e_{2}+(\dot{v}^{+}-\dot{v}^{-})\cdot n^{\varepsilon}=-\varepsilon(\bar{v}^{+,\varepsilon}-\bar{v}^{-,\varepsilon})\cdot n^{\varepsilon}
\end{equation}
and
\begin{equation}\label{5.37}
	\bar{G}^{+,\varepsilon}_{j} \cdot e_{2}  + \dot{G}^{+}_{j}\cdot  n^{\varepsilon}=-\varepsilon \bar{G}^{+,\varepsilon}_{j}\cdot n^{\varepsilon} ,~
	\bar{G}^{-,\varepsilon}_{j} \cdot e_{2}  + \dot{G}^{-}_{j}\cdot  n^{\varepsilon}=-\varepsilon \bar{G}^{-,\varepsilon}_{j}\cdot n^{\varepsilon}.
\end{equation}

Since  
$$\left\|n^{\varepsilon}(t)\right\|_{L^{\infty}}\leq \| n^{\varepsilon}\|_{H^{3}(\Gamma)}<1 $$
and $$\left\|(\bar{v}^{\varepsilon},\bar{G}^{\varepsilon}_{j})(t)\right\|_{L^{\infty}}\leq\left\|( \bar{v}^{\varepsilon},\bar{G}^{\varepsilon}_{j})(t)\right\|_{H^{3}}<1$$ is bounded uniformly, we find that
\begin{equation}\label{5.38}
	\sup _{0 \leq t \leq t_{0}}\left\|e_{2} \cdot(\bar{v}^{+,\varepsilon}-\bar{v}^{-,\varepsilon})+ (\dot{v}^{+}-\dot{v}^{-}) \cdot  n^{\varepsilon}\right\|_{L^{\infty}} \rightarrow 0 \quad \text { as } \varepsilon \rightarrow 0
\end{equation}
and
\begin{equation}\label{5.38}
	\sup _{0 \leq t \leq t_{0}} \|e_{2} \cdot \bar{G}^{+,\varepsilon}_{j} + \dot{G}^{+}_{j}\cdot  n^{\varepsilon} \|_{L^{\infty}} \rightarrow 0 \quad \text { as } \varepsilon \rightarrow 0
\end{equation}
and
\begin{equation}\label{5.38}
	\sup _{0 \leq t \leq t_{0}} \|e_{2} \cdot \bar{G}^{-,\varepsilon}_{j} + \dot{G}^{-}_{j}\cdot  n^{\varepsilon} \|_{L^{\infty}} \rightarrow 0 \quad \text { as } \varepsilon \rightarrow 0.
\end{equation}

Therefore we have 
\begin{equation}\label{5.39}
	[\bar{v}^{\varepsilon} \cdot e_{2}]= 2\dot{v}^{+}_{1} \partial_{1} f^{\varepsilon}~on~ \Gamma.
\end{equation}
and
\begin{equation}\label{5.39}
	 \bar{G}^{+,\varepsilon}_{j} \cdot e_{2}= \dot{G}^{+}_{1j} \partial_{1} f^{\varepsilon},~\bar{G}^{-,\varepsilon}_{j} \cdot e_{2}= \dot{G}^{-}_{1j} \partial_{1} f^{\varepsilon}~on~ \Gamma.
\end{equation}

We expand the third equation in \eqref{5.9} as follows
\begin{equation}\label{5.31}
	\dot{h}^{+}+ \varepsilon \bar{h}^{+, \varepsilon}=\dot{h}^{-}+ \varepsilon \bar{h}^{-,\varepsilon} ~on~ \Gamma.
\end{equation}
Since $\dot{h}^{+}=\dot{h}^{-}$, we may eliminate these two terms from equation \eqref{5.31} and divide both sides of the result equation by $\varepsilon$ to get
\begin{equation}\label{5.35}
	\bar{h}^{+,\varepsilon}=\bar{h}^{-,\varepsilon} ~on~ \Gamma.
\end{equation}

According to the bound \eqref{5.15} and sequential weak-$*$ compactness, we have that up to the extraction of a subsequence (which we still denote the subsequence by using $\varepsilon$ )
\begin{equation}\label{5.40}
	(\bar{f}^{\varepsilon}, \bar{h}^{\varepsilon}, \bar{v}^{\varepsilon},\bar{G}^{\varepsilon}_{j}) \stackrel{*}\rightharpoonup (f^{\star}, h^{\star}, v^{\star},G^{\star}_{j}) \quad \text { weakly }-* \text { in } L^{\infty}\left(\left[0, t_{0}\right] ; H^{3}(\Omega)\right) \quad \text { as } \varepsilon \rightarrow 0.
\end{equation}
By lower semicontinuity, we have
\begin{equation}\label{5.41}
	\sup _{0 \leq t \leq t_{0}}\left\|(f^{\star},h^{\star},v^{\star},G^{\star}_{j})(t)\right\|_{H^{3}} \leq 1. 
\end{equation}
In according with \eqref{5.188}, \eqref{5.24}, and \eqref{5.30}, we get
\begin{equation}\label{5.42}
	\limsup _{\varepsilon \rightarrow 0} \sup _{0 \leq t \leq t_{0}}\left\|(\partial_{t} \bar{f}^{\varepsilon}, \partial_{t} \bar{h}^{\varepsilon}, \partial_{t} \bar{v}^{\varepsilon},\partial_{t} \bar{G}^{\varepsilon}_{j})(t)\right\|_{H^{2}}<\infty.
\end{equation}

By Lions-Abin lemma in \cite{Simon}, we then have that the sequence $\left\{\left(f^{\varepsilon}, h^{\varepsilon}, v^{\varepsilon},G^{\varepsilon}_{j}\right)\right\}$ is strongly precompact in the space $L^{\infty}\left(\left[0, t_{0}\right] ; H^{8/3}(\Omega)\right)$, thus
\begin{equation}\label{5.43}
	(\bar{f}^{\varepsilon},\bar{h}^{\varepsilon}, \bar{v}^{\varepsilon},\bar{G}^{\varepsilon}_{j}) \rightarrow (f^{\star}, h^{\star}, v^{\star},G^{\star}_{j})\quad \text { strongly in } L^{\infty}\left(\left[0, t_{0}\right] ; H^{8 / 3}(\Omega)\right). 
\end{equation}
This strong convergence, together with  \eqref{5.18}, \eqref{5.24},\eqref{5.30}, implies that
\begin{equation}\label{5.44}
	(\partial_{t} \bar{f}^{\varepsilon}, \partial_{t} \bar{h}^{\varepsilon}, \partial_{t} \bar{v}^{\varepsilon}, \partial_{t} \bar{G}^{\varepsilon}_{j}) \rightarrow\left(\partial_{t} f^{\star},  \partial_{t} h^{\star},\partial_{t} v^{\star}, \partial_{t}G^{\star}_{j} \right) \text { strongly in } L^{\infty}\left(\left[0, t_{0}\right] ; H^{5 /3}(\Omega)\right),
\end{equation}
The index $\frac{8}{3}$ and $\frac{5}{3}$ are sufficient large to ensure $L^{\infty}([0,t_{0}];L^{\infty})$ convergence of $\left\{\left(f^{\varepsilon}, h^{\varepsilon}, v^{\varepsilon},G^{\varepsilon}_{j}\right)\right\}$ as $\varepsilon$ goes to zero, thus we have 
\begin{equation}\label{5.45}
	\left\{
	\begin{aligned}
		&\partial_{t}h^{\star}+(\dot{v} \cdot \nabla) h^{\star}+ \nabla \cdot v^{\star}=0& \text { in } \Omega, \\
		&\partial_{t} v^{\star}+(\dot{v} \cdot \nabla) v^{\star}+ c^{2}\nabla h^{\star}=\sum_{j=1}^{2}  (  \dot{G}_{j} \cdot \nabla )G^{\star}_{j} & \text { in } \Omega, \\ 
		&\partial_{t} G^{\star}_{j}+(\dot{v} \cdot \nabla) G^{\star}_{j}=(  \dot{G}_{j} \cdot \nabla )v^{\star}& \text { in } \Omega,\\
		&\partial_{t} f^{\star}+\dot{v}_{1} \partial_{1} f^{\star}-v^{\star}_{2}=0& \text { on } \Gamma,
	\end{aligned}
	\right.
\end{equation}
and
\begin{equation}\label{5.46}
	\begin{aligned}
	\left(v^{+,\star}-v^{-,\star}\right) \cdot e_{2}= 2\dot{v}^{+}_{1} \partial_{1} f^{\star}~~ & \text { on }~\Gamma.\\
	G^{+,\star}_{j} \cdot e_{2}= \dot{G}^{+}_{1j} \partial_{1} f^{\star},~G^{-,\star}_{j}  \cdot e_{2}= \dot{G}^{-}_{1j} \partial_{1} f^{\star}~~& \text { on }~\Gamma,\\
		h^{+,\star}=h^{-,\star}~~ & \text { on }~\Gamma.
	\end{aligned}
\end{equation}
We also pass to the limit in the initial conditions $(\bar{f}^{\varepsilon}_{0}, \bar{h}^{\varepsilon}_{0}, \bar{v}^{\varepsilon}_{0},\bar{G}^{\varepsilon}_{0,j})=(f^{L}_{0}, h^{L}_{0}, v^{L}_{0},G^{L}_{0,j})$ as $\varepsilon$ goes to 0 to obtain
\begin{equation*}
	(f^{\star}_{0}, v^{\star}_{0}, h^{\star}_{0},G^{\star}_{0,j}) =(f^{L}_{0}, h^{L}_{0}, v^{L}_{0},G^{L}_{0,j}).
\end{equation*}
Now we can see that $(f^{\star}, v^{\star}, h^{\star}, G^{\star}_{j} )(t,x)$ are solutions to \eqref{2.2} with boundary conditions \eqref{2.3} satisfying the same initial data. In according with the uniqueness result in lemma 4.1, we
have
\begin{equation}\label{5.47}
	(f^{\star}, v^{\star}, h^{\star},G^{\star}_{j})(t,x) =(f^{L}, v^{L}, h^{L},G^{L}_{j})(t,x). 
\end{equation}
Therefore  we combine  inqualities \eqref{5.41} and \eqref{5.13} to get
\begin{equation}\label{5.48}
	2=\alpha\leq \sup _{0 \leq t \leq t_{0}}\left\|(f^{\star},  h^{\star},v^{\star},G^{\star}_{j})(t)\right\|_{H^{3}} \leq 1. 
\end{equation}
which is a contradiction. Therefore, the proof of Theorem 2.3 is completed.

\phantomsection

\end{document}